\documentclass[letterpaper, 11pt,  reqno]{amsart}

\usepackage{amsmath,amssymb,amscd,amsthm,amsxtra, esint}

\usepackage[implicit=true]{hyperref}

\usepackage{color}
\usepackage{ulem}
\usepackage[makeroom]{cancel}

\allowdisplaybreaks[2]

\definecolor{gr}{rgb}   {0.,   0.69,   0.23 }
\definecolor{bl}{rgb}   {0.,   0.5,   1. }
\definecolor{mg}{rgb}   {0.85,  0.,    0.85}
\definecolor{yl}{rgb}   {0.8,  0.7,   0.}
\definecolor{or}{rgb}  {0.7,0.2,0.2}

\newtheorem{theorem}{Theorem} [section]

\newtheorem{lemma}[theorem]{Lemma}
\newtheorem{proposition}[theorem]{Proposition}
\newtheorem{remark}[theorem]{Remark}

\newtheorem{corollary}[theorem]{Corollary}

\newcommand{\I}{\hspace{0.5mm}\text{I}\hspace{0.5mm}}
\newcommand{\II}{\text{I \hspace{-2.8mm} I} }

\newcommand{\noi}{\noindent}
\newcommand{\Z}{\mathbb{Z}}
\newcommand{\R}{\mathbb{R}}

\newcommand{\T}{\mathbb{T}}

\let\Re=\undefined\DeclareMathOperator*{\Re}{Re}
\let\Im=\undefined\DeclareMathOperator*{\Im}{Im}

\let\P= \undefined
\newcommand{\P}{\mathbf{P}}

\newcommand{\E}{\mathbb{E}}

\newcommand{\F}{\mathcal{F}}

\newcommand{\al}{\alpha}
\newcommand{\be}{\beta}
\newcommand{\dl}{\delta}

\newcommand{\nb}{\nabla}

\newcommand{\Dl}{\Delta}
\newcommand{\eps}{\varepsilon}
\newcommand{\kk}{\kappa}
\newcommand{\g}{\gamma}
\newcommand{\G}{\Gamma}
\newcommand{\ld}{\lambda}
\newcommand{\Ld}{\Lambda}
\newcommand{\s}{\sigma}

\newcommand{\ft}{\widehat}

\newcommand{\wt}{\widetilde}
\newcommand{\cj}{\overline}
\newcommand{\dx}{\partial_x}
\newcommand{\dt}{\partial_t}
\newcommand{\dd}{\partial}

\newcommand{\ta}{\theta}

\renewcommand{\l}{\ell}
\renewcommand{\o}{\omega}
\renewcommand{\O}{\Omega}

\newcommand{\les}{\lesssim}
\newcommand{\ges}{\gtrsim}

\newcommand{\jb}[1]
{\langle #1 \rangle}

\renewcommand{\b}{\beta}
\newcommand{\ind}{\mathbf 1}

\newcommand{\N}{\mathbb{N}}

\newcommand{\EE}{\mathcal{E}}

\renewcommand{\H}{\mathcal{H}}

\newtheorem*{ackno}{Acknowledgements}

\numberwithin{equation}{section}
\numberwithin{theorem}{section}

\newcommand{\NLW}{\textup{NLW}}
\newcommand{\NKG}{\textup{NLKG}}

\newcommand{\gf}{\mathfrak{g}}

\begin{document}
\baselineskip = 14pt

\title[Quasi-invariant measures for 2-$d$ NLW ]
{
Quasi-invariant Gaussian measures for
 the two-dimensional defocusing cubic nonlinear wave equation
} 
\author[T.~Oh and   N.~Tzvetkov]
{Tadahiro Oh  and Nikolay Tzvetkov}

\address{
Tadahiro Oh, School of Mathematics\\
The University of Edinburgh\\
and The Maxwell Institute for the Mathematical Sciences\\
James Clerk Maxwell Building\\
The King's Buildings\\
Peter Guthrie Tait Road\\
Edinburgh\\ 
EH9 3FD\\
 United Kingdom}

\email{hiro.oh@ed.ac.uk}

\address{
Nikolay Tzvetkov\\
Universit\'e de Cergy-Pontoise\\
 2, av.~Adolphe Chauvin\\
  95302 Cergy-Pontoise Cedex\\
  France}

\email{nikolay.tzvetkov@u-cergy.fr}

\begin{abstract}
We study  the transport properties of the Gaussian measures
on Sobolev spaces under the dynamics of  the two-dimensional
defocusing cubic nonlinear wave equation (NLW).
Under some regularity condition,
we prove quasi-invariance of  the mean-zero Gaussian measures on Sobolev spaces
for the NLW dynamics.
We achieve this goal by introducing 
a simultaneous renormalization on the energy functional and its time derivative
and establishing a renormalized energy estimate in the probabilistic setting.
\end{abstract}

\maketitle

\vspace{-8mm}

\section{Introduction}
\subsection{General context}
In probability theory, 
the transport properties of 
Gaussian measures under linear and nonlinear transformations have attracted wide 
attention since the seminal work of Cameron-Martin \cite{CM}. 
In the special case of linear transformations given by 
the translation by a fixed (deterministic) vector,
Cameron-Martin provided a complete answer to this question in \cite{CM}.
 This result then formed the basis of the infinite dimensional analysis, 
 the so-called Malliavin calculus. 
In \cite{RA},  Ramer further studied the transport property of Gaussian measures
under a general nonlinear transformation on an abstract Wiener space
and gave a criterion, guaranteeing  that Gaussian measures are {\it quasi-invariant} under general transformations which are (essentially speaking) Hilbert-Schmidt perturbations of the identity. 
Here, by quasi-invariance, 
we mean that 
a measure $\mu$ on a measure space $(X, \mu)$ 
and the pushforward $T_*\mu$ of $\mu$ under  a measurable transformation $T:X\to X$, 
defined by $T_*\mu = \mu\circ T^{-1}$, are 
equivalent, namely~mutually absolutely continuous with respect to each other.

The quasi-invariance result by Ramer 
is of course more general than Cameron-Martin's result 
because it applies to general nonlinear transformations 
and it is certainly the best result one can expect in the context of general nonlinear transformations. 
In~\cite{Cru1,Cru2},
Cruzeiro  studied
flows generated by vector fields on abstract Wiener spaces
and
established 
 an abstract criterion,  guaranteeing  quasi-invariance of Gaussian measures 
 under such flows. 
We point out that  the verification of such a criterion was not carried out for concrete examples in 
\cite{Cru1,Cru2}. 
Lastly, let us  mention a generalization 
of Cruzeiro's work by Peters \cite{P1,P2}.
In particular, by exploiting the symplectic structure of the vector field, 
 he also showed that the Gaussian measure on\footnote{More precisely, 
Peters considered the Gaussian measure 
$d \mu  = Z^{-1} \exp (-\frac{1}{2} \| (u, v) \|_{H^\frac{1}{2}\times H^{-\frac{1}{2}}}^2) du dv$
on $H^\s(\T)\times H^{\s-1}(\T)$, $\s <0$, 
for which  $H^\frac{1}{2} (\T)\times H^{-\frac{1}{2}}(\T)$ is the Cameron-Martin space.
See \eqref{gauss0} below.
Note that the regularity $\frac{1}{2}$
plays an important role  in \cite{P2}
since $H^\frac{1}{2} (\T)\times H^{-\frac{1}{2}}(\T)$ is the symplectic space
for the Klein-Gordon equations, including the sine-Gordon equation.
}
  $ H^\frac{1}{2} (\T)\times H^{-\frac{1}{2}}(\T)$
is quasi-invariant under the flow of the Wick ordered sine-Gordon equation on the circle.

In the recent works \cite{Tzvet, OTz, OST}, 
we further studied the transport property of Gaussian measures under nonlinear Hamiltonian PDE dynamics
and succeeded to prove quasi-invariance of Gaussian measures on periodic functions.
In particular, 
in \cite{Tzvet},
the second author introduced a general 
strategy, combining   PDE and stochastic analysis to prove quasi-invariance of Gaussian measures
under nonlinear Hamiltonian PDE dynamics, 
thus verifying 
an assumption of the type imposed in 
  \cite{Cru1,Cru2, P1} 
for some concrete examples (without relying on a special structure of an underlying space
such as the symplectic structure  in~\cite{P2}).
In~\cite{Tzvet},  we considered the BBM-type equations 
and by exploiting energy estimates,  which are quite standard in the field of hyperbolic PDEs, 
we  established quasi-invariance of Gaussian measures on periodic functions, going beyond Ramer's result.
While it was only stated in a remark, 
similar quasi-invariance results hold for 
 the one-dimensional nonlinear wave equations (NLW) and nonlinear Klein-Gordon equations (NLKG).
In \cite{OTz, OST}, we studied 
the quasi-invariance property of Gaussian measures 
under the dynamics of the one-dimensional cubic fourth order  nonlinear Schr\"odinger equation.
By applying gauge transformations\footnote{In a recent paper \cite{OTT}, by applying a further gauge transformation,
we extended the quasi-invariance result to the cubic nonlinear Schr\"odinger equation 
with third order dispersion.} 
and (an infinite iteration of) normal form transformations,
we proved quasi-invariance of Gaussian measures,
which is optimal in terms of Sobolev regularities.

In the present paper, 
we will further develop the method of \cite{Tzvet, OTz} in the context of two-dimensional nonlinear wave equations. 
We follow the new strategy introduced by the second author in \cite{Tzvet}.
Namely, we prove the quasi-invariance property for a weighted Gaussian measure which is absolutely continuous 
with respect to the underlying Gaussian measure.
The density of such a weighted Gaussian measure  is  inspired by an energy functional associated to the equation. 
Observe that  our approach is already quite different compared to Ramer's analysis \cite{RA}.  
In a sharp contrast with the previous works \cite{Tzvet, OTz, OST},  
in this work, we need to use a {\it renormalized} energy functional. 
Such a renormalized energy is closely related to renormalizations considered in  Euclidean quantum field theory \cite{Simon}. 
On the one hand, such renormalizations often force us to work with renormalized equations.
See \cite{OTh2} in the context of two-dimensional NLW endowed with Gibbs measures.
On the other hand, 
this is {\it not} the case in our analysis;
 we are able to keep the original equation despite the use of the renormalized energy. 
This is achieved by performing a {\it simultaneous renormalization
of the energy functional and its time derivative}.
See Subsection \ref{SUBSEC:1.4} below.
 In particular, after  introducing the renormalized energy, 
we establish  a {\it renormalized energy estimate}
that is suitable for studying the dynamical property of the original equation in the probabilistic manner. 
This renormalized energy estimate is the main novelty of this work. 
As we shall see below, its proof is quite intricate and it does not result 
from purely linear Gaussian considerations unlike the previous works \cite{Tzvet, OTz}.

\subsection{Main result}

Consider the defocusing cubic nonlinear wave equation on $\T^2 = (\R/\Z)^2$:
\begin{equation}\label{NLW}
\partial_t^2 u-\Delta u+u^3=0,
\end{equation}

\noi
where $u : \T^2 \times \R\rightarrow \R$ is the unknown function. 
With $v = \dt u$, we rewrite \eqref{NLW} as the following first order system:
\begin{equation}\label{NLW-sys}
\begin{cases}
\partial_t u=v\\
 \partial_t v=\Delta u-u^3.
 \end{cases}
\end{equation}

\noi
The system \eqref{NLW-sys} is a Hamiltonian system of PDEs
with the Hamiltonian:
\begin{equation}\label{H}
H(u,v)=\frac{1}{2}\int_{\T^2}\big(|\nabla u|^2 + v^2\big)dx +\frac{1}{4}\int_{\T^2}u^4dx.
\end{equation}

\noi
It is easy to verify that,  if $(u,v)$ is a smooth solution to \eqref{NLW-sys}, then
\[\frac{d}{dt}H(u(t),v(t))=0.\]

In view of the structure of the Hamiltonian $H(u,v)$ and the properties of the linear wave equation, 
it is natural to study \eqref{NLW-sys} in the space: 
$$
{\mathcal H}^s(\T^2)\equiv H^s(\T^2)\times H^{s-1}(\T^2),
$$
where $H^s(\T^2)$ is the classical $L^2$-based Sobolev space of order $s$. 
By a classical argument (see the next section),  
one can show that \eqref{NLW-sys} is globally well-posed 
in ${\mathcal H}^\sigma(\T^2)$, $\sigma\geq 1$. 
Let us denote this global flow by $\Phi_{\NLW}(t)$, $t\in\R$.

Our main goal is to study the quasi-invariance property under $\Phi_{\NLW}(t)$ of the Gaussian measure $\mu_s$, {\it formally} defined by
\begin{align}
\begin{split}
 d \mu_s 
&   = Z_s^{-1} e^{-\frac 12 \| (u,v)\|_{{\mathcal H}^{s+1}}^2} du dv
\rule[-4mm]{0mm}{0mm}
\\
& =  Z_s^{-1} \prod_{n \in \Z^2} 
 e^{-\frac 12 \jb{n}^{2(s+1)} |\ft u_n|^2}   
  e^{-\frac 12 \jb{n}^{2s} |\ft v_n|^2}   
 d\ft u_n d\ft v_n,
\end{split}
\label{gauss0}
\end{align}

\noi
where 
  $\jb{\,\cdot\,} = (1+|\,\cdot\,|^2)^\frac{1}{2}$
and $\ft u_n$ and $\ft v_n$ denote the Fourier transforms of $u$ and $v$, respectively. 
Note that this measure is naturally associated to 
the linear wave dynamics.
In particular, $\mu_s$ is invariant under the linear wave dynamics.

We can define the measure $\mu_s$ in a rigorous manner
by viewing  it as  the induced probability measure
under the map:
\begin{equation*}
\o \in \O \longmapsto (u^\o, v^\o),
 \end{equation*}

\noi
where $u^\o$ and $v^\o$ are given by\footnote{Henceforth, we drop the harmless factor $2\pi$.} 
\begin{equation}\label{series}
u^\o(x) = \sum_{n \in \Z^2} \frac{g_n(\o)}{\jb{n}^{s+1}}e^{in\cdot x}
\qquad\text{and}\qquad
v^\o(x) = \sum_{n \in \Z^2} \frac{h_n(\o)}{\jb{n}^{s}}e^{in\cdot x}.
\end{equation}

\noi
Here, 
$\{ g_n \}_{n \in \Z^2}$ and  $\{ h_n \}_{n \in \Z^2}$
are two sequences of ``independent standard" complex-valued  Gaussian random variables 
on a probability space $(\O, \F, P)$
conditioned that  $g_{-n}=\cj{g_{n}}$,  $h_{-n}=\cj{h_{n}}$.
More precisely,  with the index set $\Ld$ defined by 
\begin{align}
 \Ld =  (\Z\times \Z_+) \cup (\Z_+\times \{0\}) \cup\{(0, 0)\},
 \label{Ld1}
\end{align}

\noi
we define 
$\{g_n, h_n\}_{n\in\Ld}$ 
to be a sequence  of independent standard 
complex-valued  Gaussian random variables 
(with $g_0, h_0$ real-valued)
and set   $g_{-n}=\cj{g_{n}}$,  $h_{-n}=\cj{h_{n}}$ for $n \in \Z^2$.

The partial sums of the series in \eqref{series} 
are a Cauchy sequence in $L^2(\Omega; \mathcal{H}^{\sigma}(\T^2))$ for every $\sigma<s$
and  therefore one can view $\mu_s$ as a probability measure on ${\mathcal H}^\sigma(\T^2)$ for a fixed $\sigma<s$.
In particular, for $s>1$,
 the flow $\Phi_{\NLW}(t)$ is well defined $\mu_s$-almost surely. 
We also point out that, for the same range of $\s$, 
the triplet $\big(\H^{s+1}(\T^2), \H^\s(\T^2), \mu_s\big)$ forms an abstract Wiener space.
See~\cite{GROSS, Kuo2}.

 We now state our main result.

\begin{theorem}\label{THM:NLW}
Let $s\geq 2$ be an even integer. Then,  $\mu_s$ is quasi-invariant under $\Phi_{\NLW}(t)$. 
\end{theorem}

We next consider the defocusing cubic nonlinear Klein-Gordon equation:
\begin{equation}\label{KG}
\partial_t^2 u-\Delta u+u+u^3=0,
\end{equation}
where $u : \T^2 \times \R \rightarrow \R$. 
As in the case of NLW, we rewrite \eqref{KG} as the first order system:
\begin{equation}\label{KG-sys}
\begin{cases}
\partial_t u=v\\
 \partial_t v=\Delta u-u-u^3.
 \end{cases}
\end{equation}

\noi
The system \eqref{KG-sys} is a Hamiltonian system of PDEs
with the Hamiltonian:
\begin{equation*}
E(u,v)=\frac{1}{2}\int_{\T^2}\big(u^2+|\nabla u|^2+v^2\big)dx +\frac{1}{4}\int_{\T^2}u^4dx
\end{equation*}

\noi
and one directly verifies that, if $(u,v)$ is a smooth solution to \eqref{KG-sys}, then
$$
\frac{d}{dt}E(u(t),v(t))=0.
$$

\noi
We again have that \eqref{KG-sys} is globally well-posed in ${\mathcal H}^\sigma(\T^2)$, $\sigma\geq 1$ (see Lemma~\ref{LEM:CP} below). 
Let us denote this global flow by $\Phi_{\NKG}(t)$, $t\in\R$.
Then, we have the following statement.

\begin{theorem}\label{THM:NLKG}
Let $s\geq 2$ be an even integer. Then,  $\mu_s$ is quasi-invariant under $\Phi_{\NKG}(t)$. 
\end{theorem}

While the proofs of Theorem~\ref{THM:NLW} and Theorem~\ref{THM:NLKG} are very similar,
it is more convenient to first prove Theorem~\ref{THM:NLKG}.
Hence, 
we shall discuss the proof of Theorem~\ref{THM:NLKG} in details and we will indicate the needed modifications leading to the proof of Theorem~\ref{THM:NLW}
in the last section of the paper.

\subsection{Remarks \& comments} \label{SUBSEC:1.3}
The restriction  that $s$ is an even integer 
in Theorems~\ref{THM:NLW} and~\ref{THM:NLKG} is not essential. 
We strongly believe that  our proof together with some classical  (in the field of dispersive PDEs) 
fractional Leibniz rule considerations provides quasi-invariance of $\mu_s$ for every $s\geq 2$.  
The extension of  Theorems~\ref{THM:NLW} and~\ref{THM:NLKG} to $s<2$ 
may also be tractable by incorporating  some of the recent development in  the low regularity 
probabilistic well-posedness of NLW and NLKG.\footnote{For example, 
the work  \cite{OTh2} on the invariant Gibbs measure for the 2-$d$ NLKG implies
quasi-invariance of $\mu_0$ under the renormalized NLKG dynamics.
For $\mu_s$ with $s > 0$, one should  not need the renormalized equation.}
 In order to highlight our renormalization argument,  
 we decided not to pursue these extensions here. 
Similarly, we believe that  our argument is applicable
to the defocusing nonlinearities of higher degrees.
For the conciseness of the  presentation, 
however, we only work with the cubic nonlinearity.
We also point out that  our argument does not extend
to the three-dimensional case. 
The proof of the main results (Theorems \ref{THM:NLW} and  \ref{THM:NLKG})  in the two-dimensional case 
is  based on a simultaneous renormalization
of the energy functional and its time derivative (Subsection \ref{SUBSEC:1.4}),
which allows us to (i)~construct
a weighted Gaussian measure associated to the renormalized energy
(Section~\ref{SEC:typical})
and (ii)~establish  a renormalized energy estimate (Theorem~\ref{THM:2}), 
 controlling the time derivative of the renormalized energy.
As we point out in  Remarks~\ref{REM:3d} and~\ref{REM:3d2}, 
both (i) and (ii) fail in the three-dimensional case.
It would be of great interest to investigate the three-dimensional case
by possibly introducing a further (simultaneous) renormalization.

In \cite{OTz, OST}, 
we studied the cubic fourth order nonlinear Schr\"odinger equation  on the circle: 
\begin{equation}
\label{4NLS0}
i \dt u = \dx^4 u + |u|^2u
\end{equation}

\noi
and proved quasi-invariance of the Gaussian measure $\nu_s$ on $L^2(\T)$ formally defined by 
\begin{align}
 d \nu_s 
  = Z_s^{-1} e^{-\frac 12 \| u\|_{H^s}^2} du 
  = Z_s^{-1} \prod_{n \in \Z} e^{-\frac 12 \jb{n}^{2s} |\ft u_n|^2}   d\ft u_n , 
\label{nu}
 \end{align}

\noi
provided that $s > \frac 12$.
In \cite{OST}, we also showed that the dispersion is essential
for this quasi-invariance result.
More precisely, we considered 
 the following dispersionless model on~$\T$:
\begin{equation}\label{ND1}
i \partial_t u = |u|^2u
\end{equation}

\noi
and showed that the Gaussian measure $\nu_s$ is {\it not} quasi-invariant
under the flow of \eqref{ND1}.
In a similar manner, we believe that 
the dispersive term is crucial 
in order to establish the quasi-invariance result  in Theorem~\ref{THM:NLW}, 
no matter how large $s$ is.  
It is quite likely that the method of \cite{OST} can be adapted to show 
that the transport of $\mu_s$ under the (well defined) flow of 
\begin{equation}\label{ODE}
\begin{cases}
\partial_t u=v\\
 \partial_t v=-u^3
 \end{cases}
\end{equation}

\noi
is  not equivalent to  $\mu_s$ (for non-trivial times). 
Indeed, we expect that the flow of \eqref{ODE} introduces fast time oscillations,  
modifying some fine regularity properties which hold true typically with respect to 
the Gaussian measure $\mu_s$.

As it is well known,
 the solutions to NLW  can be decomposed 
 as the linear evolution plus a ``one-derivative smoother term". 
 On the other hand,  the typical Sobolev regularity on the support of $\mu_s$
is  ${\mathcal H}^{\sigma}(\T^2)$, $\sigma<s$. 
The Cameron-Martin theorem in this context 
states that 
for a fixed $(h_1,h_2)\in \H^{\sigma+1}(\T^2)$,  the transport of $\mu_s$ under the shift
$$
(u,v)\longmapsto (u,v)+(h_1,h_2)
$$ 

\noi
is singular with respect to the original measure $\mu_s$. 
Therefore,  the results in Theorems~\ref{THM:NLW} and~\ref{THM:NLKG} 
represent remarkable statements, displaying fine properties of the vector fields generating $\Phi_{\NLW}(t)$ 
and $\Phi_{\NKG}(t)$. 
Moreover, we believe that the results of Theorems~\ref{THM:NLW} and~\ref{THM:NLKG} are completely out of reach of Ramer's result \cite{RA} for which we would need $(2+\eps)$-smoothing on the nonlinear term.
See \cite{Tzvet, OTz} for further discussion on this topic.

According to  \cite{Bo_proceeding}, Gel'fand asked whether, in the context of Gibbs measures for Hamiltonian PDEs, one may show the quasi-invariance of the corresponding Wiener measure by a direct method. 
Our result gives some light on Gel'fand's question because now we have a method 
to directly prove quasi-invariance of a large class of Gaussian measures  supported by functions of varying regularities
for the nonlinear wave equations. 
We should also admit that our present understanding of the corresponding question for the (more complicated) nonlinear Schr\"odinger equations is quite poor.

Our main results state that the transported measure 
 $\mu_s^t:= \Phi_{\NLW}(t)_*\mu_s$ by $\Phi_{\NLW}(t)$ (or $\Phi_{\NKG}(t)$)
  is absolutely continuous with respect to $\mu_s$.
Therefore, it has a well defined Radon-Nikodym derivative
$f(t, u, v) := \frac{d  \mu_s^t}{d\mu_s}(u, v) \in L^1(d\mu_s)$. 
It would be very interesting to obtain some further properties
of the densities $f(t,u,v)$. 
We believe that a combination of our analysis and the argument in 
\cite[Corollaire 2.2 on p.\,197]{Cru1} leads to a higher integrability
of the Radon-Nikodym derivative: $f(t,u,v)\in L^p(d\mu_s(u,v))$, $p<\infty$. See also Corollary~\ref{COR:WT} below,  where the $L^2$-integrability of the Radon-Nikodym derivative is involved. 
It also seems of interest  to establish some compactness properties in $t$ of $f(t,u,v)$ and to study the time averages of $f(t,u,v)$. 

One of the consequences of our quasi-invariance results is
the following probabilistic persistence of additional regularity (= integrability)
of the solution.
Let  $(u(0), v(0))$ be initial data distributed according to the Gaussian measure $\mu_s$.
Then, it follows from the Gaussian nature of the initial data
 that 
$(u(0), v(0))$ belongs to any Sobolev spaces 
$W^{\s,p}(\T^2)
\times W^{\s-1,p}(\T^2)$, $p \leq \infty$,  
and also to H\"older spaces $\mathcal{C}^\s(\T^2)
\times \mathcal{C}^{\s-1}(\T^2)$, where $\mathcal{C}^\s(\T^2) = B^\s_{\infty, \infty}(\T^2)$,
provided that $\s < s$. 
The quasi-invariance of $\mu_s$
guarantees  the additional regularity of the global solution $(u(t), v(t))$
in the sense that, for any $t \in \R$, the solution $(u(t), v(t))$ almost surely belongs to the same Sobolev and H\"older spaces.
Such propagation of Sobolev and H\"older regularities
for general dispersive PDEs
seems to be beyond deterministic analysis at this point.

We conclude this subsection by pointing out a  connection
of our quasi-invariance results
with wave turbulence theory \cite{ZLF, Naz}.
The main goal of wave turbulence  theory is to obtain a statistical description of the out-of-equilibrium dynamics
given by  a  nonlinear dispersive PDE (for an unknown function $u(x, t)$).
Here, 
randomness enters through the initial data $u(0)$ whose Fourier coefficients $\{\ft u_n(0)\}_{n \in \Z^d}$, 
are assumed to be independent complex-valued Gaussian 
random variables with mean zero and some variance (depending on $n$, often of the form $\jb{n}^{-\al}$).
Then, by introducing  the following two-point function:\footnote{We point out that 
if both the underlying equation and the distribution of $u(0)$ are translation invariant (in space), 
then we have 
\[\E\big[ \ft u_n(t)\cj{\ft u_m(t)}\big] = 0\]

\noi
for any $t \in \R$, 
unless $n = m$. Namely, the initial uncorrelation at time 0 propagates
for all times in the translation invariant setting.}
\begin{align}
N(n, t) = \E\big[ |\ft u_n(t)|^2\big], 
\label{X0}
\end{align}

\noi
one aims to derive an effective closed system of equations
(called the kinetic equations) for the evolution of $\{N(n, t)\}_{n \in \Z^d}$
and study its stationary solutions.
Note that the two-point functions represent
 the spectral density  of the random field $u(t)$
 and hence the kinetic equations provide evolution equations for this spectral density.

Now, let us make a connection between the study of the two-point functions 
\eqref{X0} in wave turbulence theory 
and our quasi-invariance results.
In the following, we work in a general setting, 
which applies to  the situation in our previous works \cite{Tzvet, OTz, OST, OTT} and also in this paper.
For simplicity of the presentation, we consider the scalar case.
Namely, 
let $\mu = \nu_s$ be the Gaussian measure defined in \eqref{nu}
and consider a nonlinear dispersive PDE on $\T^d$ for a scalar function $u$ (such as \eqref{4NLS0})
with random initial data $u(0) = \varphi$ distributed by $\mu$.
In particular, we have
\begin{equation}
\varphi (x) = \sum_{n \in \Z^d} \frac{\gf_n(\o)}{\jb{n}^{s}}e^{in\cdot x}, 
\label{X00}
\end{equation}

\noi
where $\{ \gf_n \}_{n \in \Z^d}$ is a  sequence of independent\footnote{In the real-valued setting,
we need to impose $\gf_{-n} = \cj{\gf_n}$ as in \eqref{series}.
See \cite{Tzvet} for example.}
standard
 complex-valued  Gaussian random variables 
on a probability space $(\O, \F, P)$.
We assume that solutions exist globally in time
and hence the solution map $\Phi(t): u(0)= \varphi \mapsto u(t)$ is well defined.
Furthermore, we assume that the Gaussian measure $\mu$
is quasi-invariant under $\Phi(t)$.
Note that this is precisely the situation in 
 \cite{Tzvet, OTz, OST, OTT}.

\begin{remark}\rm
In the setting of this paper, 
we need to transform the vector-valued solution $(u, v)$
to NLW \eqref{NLW-sys} or  NLKG \eqref{KG-sys}
into a scalar (complex-valued) function $w = \frac{1}{\sqrt 2}u + \frac{i}{\sqrt 2} \jb{\nb}^{-1} v$.
If $(u(0), v(0))$ is distributed according to the Gaussian measure $\mu_s$
in \eqref{gauss0}, 
namely they are given by the random Fourier series in \eqref{series}, 
then, by setting $\gf_n = \jb{n}^{s+1} \ft w_n(0)$, $n \in \Z^2$,  
we see that 
$\{ \gf_n \}_{n \in \Z^2}$ forms a  sequence of independent
standard  complex-valued  Gaussian random variables.
 Hence, $w(0) = \varphi$ is distributed according to the Gaussian measure $\mu = \nu_{s+1}$
 and $\varphi$ is given by the random Fourier series in \eqref{X00} (with $s$ replaced by $s+1$).
Indeed, independence of $\gf_n$ and $\gf_{-n}$, $n \ne 0$,  can be seen by writing
them as 
\begin{align*}
\gf_n & = \frac{\Re g_n - \Im h_n}{\sqrt 2} + i 
\frac{\Im g_n + \Re h_n}{\sqrt 2}, \\
\gf_{-n} &  = \frac{\Re g_n + \Im h_n}{\sqrt 2} + i 
\frac{- \Im g_n + \Re h_n}{\sqrt 2},
\end{align*}

\noi
where we used $g_{-n} = \cj{g_n}$ and $h_{-n} = \cj{h_n}$. 
Therefore, Theorems \ref{THM:NLW} and \ref{THM:NLKG}
imply that, for $s \in 2\N$,  the Gaussian measure $\mu = \nu_{s+1}$ is quasi-invariant 
under the dynamics of $w(t) = \Phi(t) w(0)$.
Here,  the solution map $\Phi(t)$ for $w$
is given by 
\[ \Phi(t) (w(0)) :=  \frac{1}{\sqrt 2} \Phi_1 (t)(u(0)) + \frac{i}{\sqrt 2} \jb{\nb}^{-1}  \Phi_2 (t)(v(0)),\]

\noi
where  $\Phi_1(t)$ and $\Phi_2(t)$ denote the first and second
components of the (vector-valued) solution map $\Phi_\text{NLW}(t)$ or $\Phi_\text{NLKG}(t)$.
\end{remark}

Under the assumptions above, we state the following corollary
to our quasi-invariance results in the general setting.
This  corollary
allows us to express the two-point functions in terms of the Radon-Nikodym derivative.

\begin{corollary}\label{COR:WT}
Let  $\mu$ be the  quasi-invariant measure under $\Phi(t)$ as above. 
We denote by  $\mu^t = \Phi(t)_* \mu$ the pushforward of $\mu$ under $\Phi(t)$
and by $\frac{d \mu^t}{d\mu}$  its Radon-Nikodym derivative. 
Suppose that $\frac{d \mu^t}{d\mu} \in L^2( d\mu)$ for some $t \in \R$.
Then, we have 
\begin{align}
N(n, t) 
= \int |\ft{\varphi}(n)|^2  \frac{d \mu^t}{d\mu}(\varphi) d\mu (\varphi)
\label{X0a}
\end{align}
 
 \noi
for any $n \in \Z^d$,  
where $N(n, t)$ is the two-point function defined in \eqref{X0}.
 
\end{corollary}

Corollary \ref{COR:WT} reduces 
the study of  the two-point functions $\{N(n, t)\}_{n \in \Z^d}$ 
in wave turbulence theory
 to  studying the dynamical property of the Radon-Nikodym
derivative
$\frac{d \mu^t}{d\mu}$.
This shows the importance of 
establishing the quasi-invariance property of the Gaussian measures
from the  viewpoint of wave turbulence theory.
It also shows the importance of establishing a higher moment bound
on the Radon-Nikodym derivative. 
Furthermore, 
by viewing $\varphi$ as 
\[\varphi : \o \in \O \mapsto \varphi^\o = \sum_{n \in \Z^d} \frac{\gf_n(\o)}{\jb{n}^s}e^{in \cdot x},\]

\noi 
we can rewrite \eqref{X0a} as 
\begin{align}
N(n, t) 
= \int_\O \frac{|\gf_n(\o)|^2 - 1}{ \jb{n}^{2s}}  \frac{d \mu^t}{d\mu}(\varphi(\o)) P(d\o)
+ \frac{1}{\jb{n}^{2s}},
\label{Y1}
\end{align}

\noi
since $\mu = P\circ \varphi^{-1}$ by definition.
%
%
Hence,  it suffices to study
the projection of the Radon-Nikodym derivative $\frac{d \mu^t}{d\mu}$
onto the subclass of the Wiener homogeneous chaoses of order two
spanned by $\{|\gf_n|^2 - 1\}_{n \in \Z^d}$.
See also Remark \ref{REM:WT1}.

\begin{proof}[Proof of Corollary \ref{COR:WT}]
By the definition of $\mu^t = \Phi(t)_* \mu$, we have
\begin{align}
\mu^t (A) = \mu (\Phi(-t) A) = \int \ind_{\{\Phi(t) \varphi \in A\}} d\mu(\varphi).
\label{X1}
\end{align}

\noi
On the other hand, we have
\begin{align}
\mu^t (A) =  \int \ind_{\{ \varphi \in A\}} d\mu^t(\varphi)
=  \int \ind_{\{ \varphi \in A\}} \frac{d \mu^t}{d\mu}(\varphi) d\mu(\varphi), 
\label{X2}
\end{align}

\noi
where the existence of the Radon-Nikodym derivative $\frac{d \mu^t}{d\mu}$
is guaranteed by the  quasi-invariance of $\mu$ under $\Phi(t)$.
Hence, from \eqref{X1} and \eqref{X2}, we obtain
\begin{align}
\int \ft{\Phi(t)\varphi}(n)\cj{\ft{\Phi(t)\varphi}(m)}  d\mu (\varphi)
= \int \ft{\varphi}(n)\cj{\ft{\varphi}(m)} \frac{d \mu^t}{d\mu}(\varphi) d\mu (\varphi).
\label{X3}
\end{align}

\noi
In particular, when $n = m$, this yields  \eqref{X0a}. 
\end{proof}

\begin{remark}\label{REM:WT1}\rm

(i)
In the setting of \cite{Tzvet, OTz, OST, OTT} and this paper, 
both  the solution map $\Phi(t)$ and the Gaussian measure $\mu$ are translation invariant (in space).
Hence, we have 
\begin{align}
\int \ft{\Phi(t)\varphi}(n)\cj{\ft{\Phi(t)\varphi}(m)}  d\mu(\varphi)
= 0 
\label{X4}
\end{align}

\noi
for $n \ne m$.
Then, it follows from \eqref{X3} and~\eqref{X4} that
\begin{align}
 \int_\O \gf_n(\o) \cj{\gf_m(\o)}\frac{d \mu^t}{d\mu}(\varphi(\o)) P(d\o)
 = 0
 \label{Y2}
\end{align}

\noi
\noi
for any $n \ne m$, 
provided that $\frac{d \mu^t}{d\mu} \in L^2( d\mu)$.
This shows that 
the projection of the Radon-Nikodym derivative $\frac{d \mu^t}{d\mu}$
onto a particular subclass of the Wiener homogeneous chaoses of order two
(i.e.~the span of $\{\gf_n \cj{\gf_m}\}_{n , m\in \Z^d, n \ne m}$)
is 0.

\smallskip

\noi
(ii) If $\jb{n}^{-2s}$ happens to describe
an invariant power spectrum for the underlying dynamics, 
namely  $N(n, t)$ is independent of time for any $n \in \Z^d$, then 
it follows from \eqref{Y1} and \eqref{Y2} that 
\begin{align*}
 \int_\O \gf_n(\o) \cj{\gf_m(\o)}\frac{d \mu^t}{d\mu}(\varphi(\o)) P(d\o)
 = \dl_{nm}, 
\end{align*}

\noi
completely determining the (time-independent) second order coefficients
of the Wiener chaos expansion of the Radon-Nikodym derivative
$\frac{d\mu^t}{d\mu}$.

\end{remark}


\subsection{Renormalized energy}
\label{SUBSEC:1.4}

We now derive the renormalized energies associated to NLKG \eqref{KG-sys}. 
As already mentioned, these renormalized energies 
and the related energy estimates are the main novelty of this work.
 Such 
 renormalizations  usually appear in the context of low regularity solutions. 
 We find it interesting that, 
in our problem, 
 even for large $s$ (very regular solutions), 
 we are obliged to appeal to a renormalization in constructing a modified energy. 
 The analysis of the Benjamin-Ono equation \cite{TV} is another example,
 where we need to use renormalizations even for regular solutions, but in a much more perturbative manner
 as  compared to the analysis in this paper. 

In the study of the transport of $\mu_s$ under the flow of \eqref{KG-sys}, we pass to the limit $N\rightarrow\infty$ in the truncated model:
\begin{equation}\label{KG-sys_N}
\begin{cases}
\partial_t u=v\\
\partial_t v=\Delta u-u-\pi_N((\pi_N u)^3),
\end{cases}
\end{equation}

\noi
where $\pi_N$ denotes the Dirichlet projector onto the frequencies $\{|n| \leq N\}$. 
Then, it is easy to see that the low frequency part $E(\pi_N u, \pi_N v)$ of the energy
and  the truncated energy: 
\begin{align}
E_N(u,v) \label{E2}
& =\frac{1}{2}\int_{\T^2}\big(u^2+|\nabla u|^2+v^2\big)dx +\frac{1}{4}\int_{\T^2}(\pi_N u)^4dx  \\
& = E(\pi_N u, \pi_N v)+ \| (\pi_N^\perp u,\pi_N^\perp v)\|_{{\mathcal H}^{1}}^2
\notag
\end{align}

\noi
are  conserved under the flow of \eqref{KG-sys_N},
where $\pi_N^\perp  = \text{Id} - \pi_N$. 
Therefore, as in the case of the untruncated NLKG \eqref{KG-sys}, 
the Cauchy problem for \eqref{KG-sys_N} is still globally well-posed in ${\mathcal H}^\s(\T^2)$, $\s\geq 1$.

Denote $\pi_N u$ and $\pi_N v$ by $u_N$ and  $v_N$, respectively.
Taking into account the definition~\eqref{gauss0} 
of the Gaussian measure $\mu_s$, it is natural to study the expression 
$$
\frac{1}{2}\frac{d}{dt}\| (u_N(t),v_N(t))\|_{{\mathcal H}^{s+1}}^2,
$$
where $(u,v)$ is a solution to the truncated NLKG \eqref{KG-sys_N}.
A direct computation yields 
\begin{align}
\frac{1}{2}\frac{d}{dt}\| (u_N(t),v_N(t))\|_{{\mathcal H}^{s+1}}^2
&  = 
\dt \bigg[\frac 12  \int_{\T^2} (J^s v_N)^2 + \frac 12 \int_{\T^2} (J^{s+1}  u_N)^2
\bigg] \notag\\ 
&  =  \int_{\T^2} (J^{2s} v_N)(-u_N^3) ,
\label{H1a}
\end{align}
where
$$
J :=  \sqrt {1 - \Dl}.
$$

\noi
In particular, when  $s=0$, the term on the right-hand side is 
$$
-\frac{1}{4}\partial_t\bigg[
\int_{\T^2}u_N^4
\bigg]
$$

\noi
and thus we recover the conservation of 
(the low frequency part of) the 
energy $E(u_N,v_N)$. 

Let $s \geq 2$ be  an even integer.
By  the Leibniz rule, we have
\begin{align}
\int_{\T^2} (J^{2s} v_N)(-u_N^3)
& =
 -3\int_{\T^2} J^sv_N J^s u_N\, u_N^2 \notag \\
& \hphantom{X}
+ \sum_{\substack{ |\al|+|\be|+|\g| \leq s\\
 |\al|,|\be|,|\g|<s
}}
c_{\al,\be,\g}
\int_{\T^2}
J^sv_N\cdot\dd^\al u_N\cdot \dd^\be u_N\cdot \dd^\g u_N
\label{H1b}
\end{align}

\noi
for some inessential constants $c_{\al,\b,\g}$.
Furthermore, recalling that $\text{vol}(\T^2) = 1$, we can write 
\begin{align}
-3\int_{\T^2} J^sv_N & J^s u_N\,  u_N^2
 = - \frac 32 \dt \bigg[\int_{\T^2} (J^su_N)^2u_N^2 \bigg]
+ 3 \int_{\T^2} (J^su_N)^2\, v_N u_N 
\notag \\
& =
  - \frac 32 \dt\bigg[ \int_{\T^2} \P_{\ne 0} [(J^su_N)^2] \,  \P_{\ne 0}[u_N^2] \bigg]
+ 3 \int_{\T^2} \P_{\ne 0}[(J^su_N)^2]\, \P_{\ne 0}[v_N  u_N]  
\notag \\
& \hphantom{X}-   \frac 32 \dt \bigg[\int_{\T^2} (J^s u_N)^2 \int_{\T^2}  u_N^2\bigg]
+   3  \int_{\T^2} (J^s u_N)^2\int  v_N  u_N ,
\label{H1}
\end{align}

\noi
where $\P_{\ne 0}$ is the projection onto non-zero frequencies:
$\P_{\ne  0} f := f - \int_{\T^2} f$.
Here, the last two terms\footnote{Namely, we have issues
at the level of {\it both} the energy and its time derivative.} on the right-hand side of \eqref{H1} are problematic because, 
in view of \eqref{series}, we have
\begin{align}
 \s_N := \E_{\mu_s}\bigg[\int_{\T^2} (J^s u_N)^2\bigg] 
= \sum_{\substack{n \in \Z^2\\|n|\leq N}} \frac{1}{\jb{n}^2}\sim \log N
\longrightarrow \, \infty
\label{sigma}
\end{align}

\noi
as $N\to\infty$ (one may also show that we have an almost sure  divergence). 
Therefore, we need to introduce a suitable  renormalization
to treat the difficulty both at the level of the $\H^{s+1}$-energy functional and its time derivative
at the same time.

With $\s_N$ defined above, we can rewrite the last two terms on the right-hand side of \eqref{H1} as 
\begin{align}
 -   \frac 32 & \dt \bigg[\int_{\T^2} (J^s u_N)^2 \int  u_N^2\bigg]
+   3  \int_{\T^2} (J^s u_N)^2\int  v_N u_N \notag\\
& = 
 -   \frac 32 \dt \bigg[\bigg(\int_{\T^2} (J^s u_N)^2-  \s_N\bigg)\int_{\T^2}  u_N^2\bigg]
+   3 \bigg( \int_{\T^2} (J^s u_N)^2 - \s_N\bigg)\int  v_N  u_N.
\label{H2}
\end{align}

\noi
Note that the term 
$$
\int_{\T^2} (J^s u_N)^2-  \s_N
$$
is now a ``good" term since, as we shall see below, we have 
$$
\bigg\|
\int_{\T^2} (J^s \pi_N u)^2-  \s_N
\bigg\|_{L^p(d\mu_s(u,v))}\leq Cp,
$$

\noi
for any finite $p \geq 2$, 
where the constant $C>0$ is independent of $p$ and $N$.
In view of the above discussion, it is now natural to define  the {\it renormalized  energy} 
 $E_{s, N}(u,v)$ by 
\begin{equation}
E_{s, N}(u,v)  = \frac 12  \int (J^s v)^2 + \frac 12 \int (J^{s+1}  u)^2
+ \frac 32  \int (J^s \pi_N u)^2  (\pi_N u)^2 - \frac 32 \s_N  \int (\pi_N u)^2.
\label{Es}
\end{equation}

\noi
By writing $E_{s, N}(u,v)$ as 
\begin{align}
E_{s, N}(u,v)  = & \ \frac 12  \int_{\T^2} (J^s v)^2 + \frac 12 \int_{\T^2}(J^{s+1}  u)^2
+ \frac 32  \int \P_{\ne 0} [(J^su_N)^2] \,  \P_{\ne 0}[u_N^2] \notag \\
& +    \frac 32 \bigg(\int_{\T^2} (J^s u_N)^2-  \s_N\bigg)\int_{\T^2}  u_N^2,
\label{H3}
\end{align}

\noi
it follows from 
\eqref{H1a}, \eqref{H1b}, \eqref{H1}, and \eqref{H2} that,   
if $(u,v)$ is a solution to \eqref{KG-sys_N},  then
we have
\begin{align}
\dt E_{s, N}(u_N,v_N) 
&  =  
 3 \int_{\T^2}  \P_{\ne 0}[ (J^s u_N)^2  ] \, \P_{\ne 0} [v_N  u_N ]
+   3 \bigg( \int_{\T^2} (J^s u_N)^2 - \s_N\bigg)\int_{\T^2}  v_N  u_N
\notag \\
& \hphantom{XX}
+ \sum_{\substack{ |\al|+|\be|+|\g|\leq s\\
|\al|,|\be|,|\g|<s
}}
c_{\alpha,\beta,\gamma}
\int_{\T^2}
J^sv_N\cdot \dd^\al u_N \cdot \dd^\be u_N \cdot \dd^\g u_N.
\label{H8}
\end{align}

Now all terms on the right-hand side of \eqref{H8} are suitable for a perturbative analysis. Here is the precise statement.

\begin{theorem}\label{THM:2}
Let $s\geq 2$ be an even integer and let us denote by $\Phi_N(t)$ the flow of  \eqref{KG-sys_N}. 
Then,   given  $r>0$,  there is a constant $C>0$ such that 
$$
\Bigg\{\int_{\{E_N(u,v)\leq r\}} 
\Big|\dt E_{s, N}(\pi_N\Phi_N(t)(u,v))\vert_{t=0} \Big|^p d\mu_s(u,v)\Bigg\}^{\frac{1}{p}}\leq Cp
$$

\noi
for every $p\geq 2$ and every $N\in \N$.
\end{theorem}

This probabilistic energy estimate on the renormalized energy $E_{s, N}$
is the main novelty of this paper.
We will present the proof of Theorem \ref{THM:2} in Section \ref{SEC:Q}.

\begin{remark}\rm
It is worthwhile to note that the introduction of the renormalization 
at the level of the energy also 
introduces a renormalization at the level of the time derivative of the energy.
Namely, by the argument above, we renormalized  both the $\H^{s+1}$-energy functional and its time derivative at the same time.
See \eqref{H2}, \eqref{H3}, and \eqref{H8}.

\end{remark}

\begin{remark}\label{REM:disp}\rm
Consider the following dispersion generalized  NLKG:
\begin{align}
 \dt^2 u + J^{2\be} u  + u^3 = 0
\label{NLKG4}
\end{align}

\noi
for $\be > 1$.
With $v = \dt u $, we can rewrite \eqref{NLKG4} 
as
\begin{equation}
\begin{cases}
\partial_t u=v\\
\partial_t v=-J^{2\be}u -u^3.
\end{cases}
\label{NLKG5}
\end{equation}

\noi
For this equation, 
we define the Gaussian measure $\mu_s^\be$ by 
\begin{align}
 d \mu_s^\be = Z_{s, \be}^{-1} \,
e^{- \frac 12 \int (J^{s+\be}  u)^2-\frac 12  \int (J^s v)^2 } du dv.
\label{mu_b}
\end{align}

\noi
Then, a typical element $(u^\o, v^\o)$ is given by the following random Fourier series:
\begin{equation*}
u^\o(x) = \sum_{n \in \Z^2} \frac{g_n(\o)}{\jb{n}^{s+\be}}e^{in\cdot x}
\qquad\text{and}\qquad
v^\o(x) = \sum_{n \in \Z^2} \frac{h_n(\o)}{\jb{n}^{s}}e^{in\cdot x},
\end{equation*}

\noi
where $\{g_n\}_{n \in \Z^2}$ and  $\{h_n\}_{n \in \Z^2}$
are as in \eqref{series}.
Then, it is easy to see that \noi
$(u^\o, v^\o)$ belongs to 
\[H^{s + \be - 1 - \eps}(\T^2) \times H^{s-1-\eps}(\T^2)\]
almost surely for any $\eps > 0$.
In particular, for $\be > 1$, we have $u \in H^s(\T^2)$ almost surely.
In fact, we have $u \in W^{s, p}(\T^2)$ for any $ p \leq \infty$ almost surely.
This implies that 
$\int_{\T^2}  (J^s u)^2 u^2  < \infty$
almost surely
and hence there is no need to introduce a renormalized energy.
See Appendix~\ref{SEC:disp}.

Therefore, when $\be >1$,  one can proceed as in~\cite{Tzvet}
and prove quasi-invariance of 
$\mu_s^\be$ under the flow of~the dispersion generalized NLKG \eqref{NLKG4}.
In particular, when $\be = 2$,
\eqref{NLKG4} corresponds to the nonlinear beam equation on $\T^2$,
which is the borderline case for Ramer's argument
on $\T^2$
(namely, ~still non-trivial).
The same remark applies to 
the dispersion generalized NLW:
\begin{align*}
& \dt^2 u + (-\Dl)^{\be} u  + u^3 = 0.
\end{align*}

\end{remark}

\subsection{Organization of the remaining part of the manuscript}

We complete this section by introducing some notations. 
In the next section, 
we present the well known arguments assuring the existence of well-defined dynamics
in $\H^\s(\T^2)$, $\s \geq 1$. 
In Section~\ref{SEC:typical}, 
we define a weighted Gaussian measure  absolutely continuous with respect to $\mu_s$.
This weighted Gaussian measure is 
adapted to the renormalized energy $E_{s, N}$
and its  transport with respect to the truncated NLKG dynamics  $\Phi_{N}(t)$ 
is easier to handle.
Section~\ref{SEC:Q} will be devoted to the proof of Theorem~\ref{THM:2}. 
In Section~\ref{SEC:proof_th1}, 
we employ the arguments essentially introduced in our previous works \cite{Tzvet, OTz} to complete the proof of Theorem~\ref{THM:NLKG} for NLKG.
The last section is devoted to the extension of Theorem~\ref{THM:NLKG} to the case 
of the ``usual"  nonlinear wave equation (Theorem~\ref{THM:NLW}). 
In Appendix \ref{SEC:disp}, 
we briefly discuss the case of the dispersion generalized NLKG.

\subsection{Notation}
For a multi-index $\alpha=(\alpha_1,\alpha_2)\in\Z_{\geq 0}^2$, 
we set $|\alpha|=\alpha_1+\alpha_2$. 
For a frequency  $n=(n_1,n_2)\in\Z^2$, 
we set  $|n|=(n_1^2+n_2^2)^{\frac{1}{2}}$ and $\jb{n}=
(1+n_1^2+n_2^2)^{\frac{1}{2}}$.

Given  $N\in \N$, we denote  the projectors $\P_N$ and $\pi_N$  by 
$$
(\P_N u)(x)=
\sum_{N\leq \jb{n} < 2N}  \ft{ u}_{n}   \, e^{in\cdot x}
$$
and 
$$
(\pi_N u)(x)=
\sum_{ |n| \leq N}  \ft{ u}_{n}  \, e^{in\cdot x}.
$$

\noi
We also set
\[\pi_N^\perp  = \text{Id} - \pi_N. \]

\noi
We will consider the Littlewood-Paley decomposition of the form
$$
u=  \sum_{ N \geq 1, \text{\,dyadic}} \P_N u.
$$

Given $r>0$, we define $\mu_{s,N, r}$ as 
\begin{align}
d \mu_{s,N, r}(u,v)=  \ind_{\{E_N(u,v) \leq r\}} \,  d \mu_s(u,v),
\label{gauss4}
\end{align}

\noi
where $E_N(u, v)$ is the conserved energy for the truncated NLKG dynamics defined in \eqref{E2}.
Note that we do not normalize $\mu_{s, N, r}$ to be a probability measure.
We also set $\mu_{s, r} = \mu_{s, \infty, r}$.

Given $R>0$ and $\sigma\in \R$, we define 
the ball $B_{R, \s} \subset \H^\s(\T^2)$ by  
$$
B_{R,\sigma}= \big\{ (u,v)\in {\mathcal H}^\sigma(\T^2)\,:\, \|(u,v)\|_{{\mathcal H}^\sigma}\leq R\big\}.
$$

\section{On the well-posedness and approximation property 
of the truncated NLKG dynamics}
\label{SEC:2}

In this section, we briefly go over the well-posedness theory
of  the following Cauchy problem for the truncated NLKG:
\begin{equation}\label{KG-sys_CP}
\begin{cases}
\partial_t u=v\\
 \partial_t v=\Delta u-u-\pi_N((\pi_N u)^3)\\
 ( u, v)|_{t = 0} = (u_0, v_0), 
 \end{cases}
\end{equation}

\noi
where $N\geq 1$. We also allow $N=\infty$ with the convention $\pi_{\infty}={\rm Id}$. 
We have the following (well-known) result.

\begin{lemma}\label{LEM:CP}
Let $\s \geq 1$ and $N \in \N \cup \{\infty\}$.
Then, the truncated NLKG \eqref{KG-sys_CP} is globally well-posed in $\H^\s(\T^2)$.
Namely, given any  $(u_0,v_0)\in {\mathcal H}^\sigma(\T^2)$, 
there exists a unique global solution to \eqref{KG-sys_CP} in $C(\R; {\mathcal H}^\s(\T^2))$
and,  moreover, the dependence on initial data is continuous. 
If we denote by $\Phi_N(t)$ the  data-to-solution map at time $t$, then $\Phi_N(t)$ is a continuous bijection on ${\mathcal H}^\s(\T^2)$ for every $t\in\R$, satisfying the semigroup property:
$$
\Phi_N(t+\tau)=\Phi_N(t)\circ \Phi_N(\tau) 
$$

\noi
for any $t, \tau \in \R$.
\end{lemma}

When  $N=\infty$, 
 we simply denote $\Phi_{\infty}(t)= \Phi_\NKG(t)$
by $\Phi(t)$ in the following.

\begin{proof}
By rewriting \eqref{KG-sys_CP} in the Duhamel formulation, we have
\begin{equation}\label{int_eq}
(u(t),v(t))=\bar{S}(t)(u_0,v_0)+\big(F_1(u)(t),F_2(u)(t)\big),
\end{equation}

\noi
where 
$$
\bar{S}(t)(u_0, v_0) =(S(t)(u_0,v_0), \partial_t S(t)(u_0,v_0))
$$

\noi
with
\begin{align*}
S(t)(u_0,v_0)& =\cos(tJ)u_0+J^{-1}\sin(tJ)v_0,\\
\partial_t S(t)(u_0,v_0)& =-J \sin(tJ)u_0+\cos(tJ)v_0,
\end{align*}

\noi
and
\begin{align*}
F_1(u)(t)& =-\int_{0}^t J^{-1}\sin((t-\tau)J)\, \pi_N\big(
 (\pi_N u)^3\big)(\tau)d\tau,\\
F_2(u)(t)& =-\int_{0}^t \cos((t-\tau)J)\, \pi_N\big(
 (\pi_N u)^3\big)(\tau) d\tau.
\end{align*}

\noi
By a fixed point argument with the Sobolev embedding, 
one can easily solve  \eqref{int_eq} locally in time  
in $C([-T,T]; {\mathcal H}^\s(\T^2))$ for some small  $T = T\big(\|(u_0,v_0)\|_{{\mathcal H}^1}\big)>0$. 
This claim immediately follows
from the boundedness (in fact, unitarity)
of  $\bar{S}(t)$  on ${\mathcal H}^\sigma(\T^2)$ for all $\sigma\in\R$ 
and 
\begin{equation}\label{nonlinear}
\big\|\big(F_1(u)(t),F_2(u)(t)\big)\big\|_{{\mathcal H}^\sigma(\T^2)}\les \|u\|_{H^\sigma(\T^2)}\|u\|_{H^1(\T^2)}^2
\end{equation}

\noi
 for any $\s\geq 1$.
The tame  estimate \eqref{nonlinear} is a consequence of the fractional Leibniz  rule: 
\[
\|J^{\s-1} (u^3)\|_{L^2(\T^2)}\les \|J^{\s-1} u\|_{L^6(\T^2)}\|u\|_{L^6(\T^2)}^2
\]

\noi
and the Sobolev embedding:~$H^1(\T^2)\subset L^6(\T^2)$
and  ensures that the local existence time depends only on 
$\|(u_0,v_0)\|_{{\mathcal H}^1}$. 
The conservation of the truncated energy $E_N(u, v)$ defined in \eqref{E2}
provides an a priori bound on
$\|(u(t),v(t))\|_{{\mathcal H}^1}$, allowing us 
to iterate the local existence result and extend the local solutions globally in time. 
The flow properties are a standard consequence of the time reversibility of \eqref{KG-sys_CP}.
This completes the proof of Lemma~\ref{LEM:CP}.
\end{proof}

\begin{remark} \label{REM:3d0}
\rm
Note that  Lemma~\ref{LEM:CP} also holds in the three-dimensional case
because we also have the Sobolev embedding $H^1(\T^3)\subset L^6(\T^3)$.

\end{remark}

We also have the following approximation property of the truncated dynamics
\eqref{KG-sys_CP}.

\begin{lemma}\label{LEM:approx}
Let $\sigma\geq 1$, $t_0\in\R$,  and $K$ be a compact set in $\H^\s(\T^2)$. 
Then,  for every $\eps>0$,
there exists  $N_0\in \N$ such that 
\begin{equation*}
\|\Phi(t)(u,v)-\Phi_N(t)(u,v)\|_{\H^\sigma(\T^2)}<\eps
\end{equation*}

\noi
for any $t \in [0, t_0]$, any $(u,v)\in K$,  and any $N\geq N_0$
and hence 
\begin{equation*}
\Phi(t)(K)\subset \Phi_N(t)(K+B_{\eps,\sigma}).
\end{equation*}

\noi
for  any $t \in [0, t_0]$ and
any $N\geq N_0$.

\end{lemma}

The proof of Lemma~\ref{LEM:approx} is based on the identity 
$$
u^3-\pi_N\big(
(\pi_N u)^3
\big)=
\pi_N^\perp(u^3)+\pi_N\big(u^3-(\pi_N u)^3\big)
$$

\noi
and the estimates  in the proof of Lemma~\ref{LEM:CP}. 
 In our previous works~\cite{Tzvet, OTz}, 
 we presented the details of the approximation argument
 analogous to 
  Lemma~\ref{LEM:approx}
 and thus we omit details.

\section{Weighted Gaussian measure
associated to the renormalized energy}\label{SEC:typical}

In this section, we construct a weighted Gaussian measure $\rho_{s, N, r}$
associated
 to the renormalized energy $E_{s, N}$ introduced in Subsection~\ref{SUBSEC:1.4}.
We will study its transport properties in Section~\ref{SEC:proof_th1}.

Let  $r>0$ and $N\geq 1$.
In view of \eqref{gauss0} and \eqref{Es}, 
 we define a weighted Gaussian measure $\rho_{s, N, r}$  by 
\begin{align}
d\rho_{s, N, r}(u,v) 
& = \text{``}Z_{s,N, r}^{-1}\ind_{\{E_N(u,v) \leq r\}}e^{-E_{s, N}(u, v) }du dv\text{''}\notag\\
& = Z_{s,N, r}^{-1} \ind_{\{E_N(u,v) \leq r\}}e^{-  R_{s, N}(\pi_{N}u)} d \mu_s(u,v) ,
\label{K0}
\end{align}

\noi
where $E_N(u,v)$ is the conserved energy for the truncated NLKG defined in \eqref{E2}
and $R_{s, N}(u)$ is defined by 
\begin{align}
R_{s, N}(u) 
& = 
 \frac 32  \int_{\T^2} (J^s u)^2  u^2 - \frac 32 \s_N  \int_{\T^2} u^2.
\label{K1}
\end{align}

Our goal in this section is to prove the  following statement.

\begin{proposition}\label{PROP:QFT}
Let $ s > 0$  and $r>0$.
Then, given $ p < \infty$, 
there exists $C > 0$ such that 
\begin{equation}\label{vtoro}
 \Big\|\ind_{\{E_N(u,v) \leq r\}}e^{-  R_{s, N}(\pi_{N}u)}\Big\|_{L^p(d\mu_s(u,v))}\leq C
\end{equation}

\noi
for every $N\geq 1$.
Moreover, there exists $R_s(u)\in L^p(d\mu_s(u,v))$ such that 
\begin{equation}\label{parvo}
\lim_{N\rightarrow\infty}
R_{s, N}(\pi_{N}u)=R_s(u)\qquad \text{in } L^p(d\mu_s(u,v))
\end{equation}
and 
\begin{equation}\label{treto}
\lim_{N\rightarrow\infty}\ind_{\{E_N(u,v) \leq r\}}e^{-  R_{s, N}(\pi_{N}u)}=\ind_{\{E(u,v) \leq r\}}e^{-  R_s(u)}
\qquad \text{in } L^p(d\mu_s(u,v)).
\end{equation}
\end{proposition}

Proposition \ref{PROP:QFT} allows us 
to define the limiting weighted Gaussian measure $\rho_{s,  r}$ by 
\begin{align}
d\rho_{s, r}(u,v) 
& = Z_{s, r}^{-1} \ind_{\{E(u,v) \leq r\}}e^{-  R_s(u)} d \mu_s(u,v).
\label{QFT1}
\end{align}

\noi
Moreover, we have  the following `uniform convergence' property of $\rho_{s, N, r}$
to $\rho_{s, r}$;  given any   $\eps > 0$, there exists $N_0 \in \N$ such that 
\begin{align}
 |  \rho_{s,  r}(A) - \rho_{s, N, r}(A)| < \eps 
\label{rho_s}
\end{align}

\noi
for any $N \geq N_0$
and  any measurable set 
$A \subset \H^\s(\T^2)$, $\s < s$.

In the following, we first state several lemmas.
We  then present the proof of Proposition~\ref{PROP:QFT}
at the end of this section.
We first recall the following Wiener chaos estimate
\cite[Theorem~I.22]{Simon}.
See also \cite[Proposition~2.4]{TTz}.

\begin{lemma}\label{LEM:hyp}
 Let $\{ \gf_n\}_{n \in \N }$ be 
 a sequence of  independent standard real-valued Gaussian random variables.
Given  $k \in \mathbb{N}$, 
let $\{P_j\}_{j \in \N}$ be a sequence of monomials in 
$\bar  \gf = \{ \gf_n\}_{n \in \N }$ of  degree at most $k$, namely, 
$P_j = P_j (\bar \gf)$ is of the form
$P_j = c_j \prod_{i = 1}^{k_j} \gf_{n_i} $
with $k_j \leq k$ and $n_1, \dots, n_{k_j} \in \N$.
Then, for $p \geq 2$, we have
\begin{equation*}
 \bigg\|\sum_{j \in \N} P_j(\bar \gf) \bigg\|_{L^p(\O)} \leq (p-1)^\frac{k}{2} \bigg\|\sum_{j \in \N} P_j(\bar \gf) \bigg\|_{L^2(\O)}.
 \end{equation*}

\end{lemma}

This lemma is a direct corollary to the
  hypercontractivity of the Ornstein-Uhlenbeck
semigroup due to Nelson \cite{Nelson2}.
Note that in the definition of $P_j$ above, 
we may have $n_i = n_\l$ for $i \ne \l$.
Namely, we do not impose independence of the factors $\gf_{n_i}$ of $P_j$
in Lemma \ref{LEM:hyp}.
In the following, we apply Lemma \ref{LEM:hyp}
to multilinear terms involving 
$\{ g_n \}_{n \in \Z^2}$ and  $\{ h_n \}_{n \in \Z^2}$
 in \eqref{series}
by first expanding $g_n$ and $h_n$ into their 
real and imaginary parts.

We use Lemma \ref{LEM:hyp} to prove the following two lemmas.
The first lemma is a direct consequence of the linear Gaussian bound
and will be used in  Section~\ref{SEC:Q}.

\begin{lemma}\label{LEM:kin}
Let $s > 1$. Let $\alpha, \be$ be multi-indices such that $|\alpha|\leq s$ and $|\be|\leq s-1$. 
Then, for every $\delta>0$, there exists  $C>0$ such that 
\begin{align}
\big\|\|\partial^\alpha \P_{M}\pi_N u\|_{L^\infty(\T^2)}\big\|_{L^{p}(d\mu_{s}(u, v))} & \leq C\sqrt{p}M^\delta,
\label{kin0}\\
\big\|\|\partial^\be \P_{M}\pi_N v\|_{L^\infty(\T^2)}\big\|_{L^{p}(d\mu_{s}(u, v))}& \leq C\sqrt{p}M^\delta,
\label{kin0a}
\end{align}

\noi
for any $p \geq 2$ and any $N, M \in \N$.
\end{lemma}

\begin{proof}
In the following, we only prove \eqref{kin0} since \eqref{kin0a}
follows in a similar manner.
Let $q\gg 1$ be such that $q>2/\delta$. 
Then, by the Sobolev embedding $W^{\delta,q}(\T^2)\subset L^\infty(\T^2)$, it suffices to prove the bound 
\begin{align*}
\big\| \| J^\delta \partial^\alpha \P_{M}\pi_N u\|_{L^q(\T^2)}\big\|_{L^{p}(d\mu_{s}(u, v))}\leq C\sqrt{p}M^\delta\, .
\end{align*}

\noi
Without loss of generality, assume $p \geq q$.
By  Minkowski's inequality, 
we see that  it suffices to prove 
\begin{align}
\big\| \| J^\delta \partial^\alpha \P_{M}\pi_N u\|_{L^{p}(d\mu_{s}(u, v))}\big\|_{L^q(\T^2)}\leq C\sqrt{p}M^\delta\, .
\label{kin1}
\end{align}

Noting that 
$$
 \| J^\delta \partial^\alpha \P_{M}\pi_N u\|_{L^{p}(d\mu_{s}(u, v))}
 =\bigg\| \sum_{ \substack { M\leq \jb{n} < 2M \\ |n|\leq N} }
\frac{(in)^\alpha \jb{n}^\delta g_n}{\jb{n}^{s+1}}\, e^{in\cdot x}
\bigg\|_{L^p(\Omega)}, 
$$

\noi
it follows from Lemma \ref{LEM:hyp} 
that 
\begin{align*}
\bigg\| \sum_{ \substack { M\leq \jb{n} < 2M \\ |n|\leq N }   }
\frac{(in)^\alpha \jb{n}^\delta g_n}{\jb{n}^{s+1}}\, e^{in\cdot x}\bigg\|_{L^p(\Omega)}
& \leq 
\sqrt{p}\,
\bigg\| \sum_{ \substack { M\leq \jb{n} < 2M \\ |n|\leq N }   }
\frac{(in)^\alpha \jb{n}^\delta g_n}{\jb{n}^{s+1}}\, e^{in\cdot x}
\bigg\|_{L^2(\Omega)}\notag\\
& =  \sqrt{p}\,\bigg(\sum_{ \substack { M\leq \jb{n} < 2M \\ |n|\leq N }   }
\frac{|n|^{2|\alpha|}
 \jb{n}^{2\delta}
}{\jb{n}^{2(s+1)}}\bigg)^\frac{1}{2}
\leq C \sqrt{p}M^{\delta}, 
\end{align*}

\noi
yielding \eqref{kin1}.
This completes the proof of Lemma~\ref{LEM:kin}. 
\end{proof}


Set 
\begin{align}
F_{N}(u)\equiv R_{s, N}(\pi_{N}u).
\label{K2}
\end{align}

\noi
The following lemma on the convergence property of $F_N(u)$ is  inspired by the consideration in \cite{BO96}. 
Similar  analysis also appears  in the quantum field theory literature.

\begin{lemma}\label{LEM:NM}
Let $s > 0$.  Then, there exist  $\ta>0$  and $C>0$ such that 
\begin{align*}
\|F_N(u)-F_M(u)\|_{L^p(d\mu_s(u,v))} \leq Cp^2 M^{-\ta}
\end{align*}

\noi
for any $N\geq M\geq 1$ and any $p\geq 2$.
\end{lemma}

\begin{remark}\label{REM:hyp}\rm
As a corollary to Lemma \ref{LEM:NM}, we have the following tail estimate: 
\begin{align*}
\mu_s\big( (u, v) : |F_N(u)-F_M(u)| > \al \big) \leq C e^{-c  M^\frac{\theta}{2}\al^\frac{1}{2}},
\end{align*}

\noi
which follows from Lemma \ref{LEM:NM} and Chebyshev's inequality.
See also 
 \cite[Lemma~4.5]{Tz_ptrf}.
 
\end{remark}

\begin{proof}
Write
\begin{align}
\frac{3}{2} \int_{\T^2} (J^s \pi_N u)^2 (\pi_N u)^2 
=
\frac{3}{2}
\sum_{\G_N}
\jb{n_1}^{s}\jb{ n_2}^{s}
\,
\ft{ u}_{n_1}
\ft{ u}_{n_2}
\ft{ u}_{n_3}\ft{u}_{n_4}.
\label{hyp1}
\end{align}

\noi
where
$\G_N$ is defined by 
\[ \G_N = \big\{ (n_1, n_2, n_3, n_4) \in \Z^8:
n_1 + n_2+n_3+n_4=0, |n_j|\leq N\big\}.\]

\noi
We say that we have a {\it pair} if we have $n_j = -n_k$, $j \ne k$ in the summation above.
Under the condition $n_1 + n_2+n_3+n_4=0$, we have either two pairs or no pair.
We now split the summation in three cases.
(i) The first contribution comes from the case 
\[\Ld_1 = \G_N \cap \{ n_1=-n_2\},\]

\noi 
(ii) the second contribution comes from  
\[\Ld_2 = \G_N \cap \{n_1=-n_3\text{ or }n_1=-n_4 \text{ but }n_1 \ne -n_2\},\]

\noi
and (iii)  the third contribution comes from the ``no pair'' case:
\[\Ld_3 = \G_N \cap \{n_1\ne -n_j, j = 2, 3, 4\}.\]

\noi
Therefore, recalling that $\ft u_{-n} = \cj{\ft u_n}$, we have the decomposition
$$
\frac{3}{2} \int_{\T^2} (J^s \pi_N u)^2 (\pi_N u)^2 
=
J_{1,N}(u)+J_{2,N}(u)+J_{3,N}(u),
$$
 
\noi
where 
$J_{j,N}(u)$, $j = 1, 2, 3$, 
is the contribution to \eqref{hyp1} from $\Ld_j$, 
satisfying
\begin{align}
J_{1,N}(u)& =\frac{3}{2}\Big(
\sum_{|n|\leq N}\jb{n}^{2s}|\ft{ u}_{n}|^2
\Big)
\Big(\sum_{|n|\leq N}|\ft{ u}_{n}|^2\Big), \notag \\
J_{2,N}(u)
& =3\sum_{|n|\leq N}
\jb{n}^{s}|\ft{ u}_{n}|^2
\Big(\sum_{\substack{|m|\leq N\\m \ne n }}
\jb{m}^{s}|\ft{ u}_{m}|^2\Big)
- \frac32 \sum_{\substack{|n|\leq N\\n \ne 0}}
\jb{n}^{2s}|\ft{ u}_{n}|^4, \label{hyp1X}\\
J_{3,N}(u)
& =\frac{3}{2}\sum_{\Ld_3}
\jb{n_1}^{s}\jb{n_2}^{s} \ft{ u}_{n_1}\ft{ u}_{n_2}\ft{ u}_{n_3}\ft{u}_{n_4}.\notag 
\end{align}

\noi
Note that 
the first term in \eqref{hyp1X}
corresponds to the contribution
from 
\[ \{n_1=-n_j \text{ but }n_1 \ne -n_2\},\quad j = 3, 4.\]

\noi
We, however, needed to subtract the 
contribution from 
\[ \{n_1=-n_3 = -n_4\text{ but }n_1 \ne -n_2\}, \]

\noi
which was counted twice.
This corresponds to the second term in \eqref{hyp1X}.
Note that we need the restriction $n \ne 0$
since $n_1 \ne - n_2$.

Now, 
by setting 
$$
\wt J_{1,N}(u)=
J_{1,N}(u)- 
\frac 32 \s_N  \int_{\T^2} (\pi_N u)^2
= 
\frac{3}{2}
\bigg(\Big(\sum_{|n|\leq N}
\jb{n}^{2s}|\ft{ u}_{n}|^2\Big) -\s_N\bigg)
\Big(\sum_{|n|\leq N}|\ft{ u}_{n}|^2\Big),
$$

\noi
it suffices to prove the following three estimates:
\begin{align}
\label{McD1}
\big\| \wt J _{1,N}(u^\o)-\wt J_{1,M}(u^\omega)\big\|_{L^p(\Omega)}\les p^2M^{-\ta},\\
\label{McD2}
\big\|  J _{2,N}(u^\o)- J_{2,M}(u^\omega)\big\|_{L^p(\Omega)}\les p^2M^{-\ta},\\
\label{McD3}
\big\|  J _{3,N}(u^\o)- J_{3,M}(u^\omega)\big\|_{L^p(\Omega)}\les p^2M^{-\ta},
\end{align}

\noi
where 
$u^\omega$ is as in  \eqref{series}.

With  the definition \eqref{sigma} of $\s_N$, 
the  left-hand side of \eqref{McD1} equals
\begin{align}
\frac{3}{2}
\bigg\|
\bigg(\sum_{|n|\leq N}
\frac{|g_n|^2-1}{\jb{n}^2}
\bigg)\bigg(\sum_{|n|\leq N}
& \frac{|g_n|^2}{\jb{n}^{2(s+1)}}
\bigg)   \notag\\
& -\bigg(\sum_{|n|\leq M}
\frac{|g_n|^2-1}{\jb{n}^2}
\bigg)
\bigg(
\sum_{|n|\leq M}
\frac{|g_n|^2}{\jb{n}^{2(s+1)}}
\bigg)
\bigg\|_{L^p(\O)}
\label{kan1}
\end{align}

\noi
Then, with
\begin{align}
A_N B_N-A_M B_M=(A_N-A_M)B_N +A_M(B_N-B_M), 
\label{hyp1a}
\end{align}

\noi
we can estimate \eqref{kan1} by
\begin{align*}
C\bigg\|
\bigg(
\sum_{M<|n|\leq N}
& \frac{|g_n|^2-1}{\jb{n}^2}
\bigg)
 \bigg(\sum_{|n|\leq N} \frac{|g_n|^2}{\jb{n}^{2(s+1)}}
\bigg)\bigg\|_{L^p(\O)}
\\
& +
C\bigg\|
\bigg(
\sum_{|n|\leq N}
\frac{|g_n|^2-1}{\jb{n}^2}
\bigg)
\bigg(
\sum_{M<|n|\leq N}
\frac{|g_n|^2}{\jb{n}^{2(s+1)}}
\bigg)
\bigg\|_{L^p(\Omega)}
=: \I + \II.
\end{align*}

We now estimate $\I$ and $\II$. 
By H\"older's inequality, Lemma \ref{LEM:hyp}, 
and the triangle inequality, we have 
\begin{align}
\I
& \les \bigg\|\sum_{M<|n|\leq N}
\frac{|g_n|^2-1}{\jb{n}^2}\bigg\|_{L^{2p}(\O)}
\bigg\|
\sum_{|n|\leq N}
\frac{|g_n|^2}{\jb{n}^{2(s+1)}}
\bigg\|_{L^{2p}(\O)}\notag \\
& \les
p \, \bigg\|\sum_{M<|n|\leq N}
\frac{|g_n|^2-1}{\jb{n}^2}\bigg\|_{L^{2}(\Omega)}
\sum_{|n|\leq N}
\frac{\|g_n\|_{L^{4p}(\O)}^2}{\jb{n}^{2(s+1)}}.
\label{hyp1b}
\end{align}

\noi
Noting that $\|g_n\|_{L^{4p}(\O)}\les \sqrt{p}$ and 
\begin{align}
\label{hyp1c}
\E\big[(|g_n|^2-1)(|g_m|^2-1)\big] = 0 
\end{align}

\noi
unless $n = \pm m$, 
we obtain 
\begin{align}
\I\les p^2\bigg(\sum_{M<|n|\leq N}\frac{1}{\jb{n}^4}\bigg)^\frac{1}{2}
\bigg(\sum_{|n|\leq N}
\frac{1}{\jb{n}^{2(s+1)}}
\bigg)
\les p^2M^{-1}.
\label{hyp2}
\end{align}

Next, we estimate $\II$.
Proceeding as above, we obtain 
\begin{align}
\II\les p^2\bigg(\sum_{|n|\leq N}\frac{1}{\jb{n}^4}\bigg)^\frac{1}{2}
\bigg(\sum_{M<|n|\leq N}
\frac{1}{\jb{n}^{2(s+1)}}
\bigg)
\les p^2M^{-2s}.
\label{hyp3}
\end{align}
\noi

\noi
Hence, \eqref{McD1} follows from \eqref{hyp2} and \eqref{hyp3}
provided that $\theta \leq \min(1, 2s)$.

Let us next turn to the proof of \eqref{McD2}. 
By  the triangle inequality, $\|g_n\|_{L^{4p}(\O)}\les \sqrt{p}$, and 
\eqref{hyp1a}, 
we have
\begin{align*}
\text{LHS of } \eqref{McD2}
& \les p^2\Bigg\{ \bigg(\sum_{M<|n|\leq N}\frac{1}{\jb{n}^{s+2}}\bigg)
\bigg(\sum_{|n|\leq N}\frac{1}{\jb{n}^{s+2}}\bigg)
+ \sum_{M<|n|\leq N}\frac{1}{\jb{n}^{2s+4}}\Bigg\}\\
& \leq Cp^2M^{-s},
\end{align*}

\noi
provided that $s > 0$.
This proves \eqref{McD2}.

Let us finally  turn to \eqref{McD3}. 
In this case, it suffices to prove
\begin{equation}\label{McD3_bis}
\bigg\|\sum_{\substack{\Ld_3 \\
\max|n_j|>M}}
\frac{g_{n_1}g_{n_2}g_{n_3}g_{n_4}}
{\jb{n_1}\jb{n_2}\jb{n_3}^{s+1}\jb{n_4}^{s+1}}\bigg\|_{L^p(\Omega)}
\les p^2M^{-\ta}.
\end{equation}

\noi
By Lemma \ref{LEM:hyp}, 
the left-hand side of \eqref{McD3_bis} is bounded by
\begin{align}
p^2
& \bigg\|\sum_{\substack{\Ld_3 \\
\max|n_j|>M}}
\frac{g_{n_1}g_{n_2}g_{n_3}g_{n_4}}
{\jb{n_1}\jb{n_2}\jb{n_3}^{s+1}\jb{n_4}^{s+1}}\bigg\|_{L^2(\Omega)}\notag\\
& = p^2\,  \E\Bigg[\bigg(\sum_{\substack{\Ld_3 \\
\max|n_j|>M}}
\frac{g_{n_1}g_{n_2}g_{n_3}g_{n_4}}
{\jb{n_1}\jb{n_2}\jb{n_3}^{s+1}\jb{n_4}^{s+1}}\bigg)\notag\\
& 
\hphantom{XXXXXXX}
\times \bigg(\sum_{\substack{\Ld_3 \\
\max|m_j|>M}}
\frac{\cj{g_{m_1}g_{m_2}g_{m_3}g_{m_4}}}
{\jb{m_1}\jb{m_2}\jb{m_3}^{s+1}\jb{m_4}^{s+1}}\bigg)\Bigg]^\frac{1}{2}.
\label{hyp4}
\end{align}

\noi
Recalling that
\[ \E\big[ g_n^k \cj{g_m}^\l\big] = \dl_{nm} \dl_{k\l}\cdot  k! \]

\noi
(for $n, m \ne 0$),\footnote{Recall that  $g_0$ is real-valued
and thus we have $\E[g_0^{2k}] = \frac{(2k)!}{2^k\cdot k!}$.}
we see that the non-zero  contribution to \eqref{hyp4} comes from 
$m_j = n_{\s(j)}$, $j = 1, \dots, 4$, 
for some permutation $\s \in S_4$.
Hence,  we have 
\begin{align}
\eqref{hyp4}
& \les p^2 \Bigg[\sum_{\substack{\G_N \\ \max|n_j|>M}}
\bigg(
& \frac{1}
{\jb{n_1}^2 \jb{n_2}^2 \langle n_3\rangle^{2(s+1)}\langle n_4\rangle^{2(s+1)}}
 +
\frac{1}{\prod_{j = 1}^4 \jb{n_j}^{s+2}}
\bigg)\Bigg]^\frac{1}{2} \notag\\
& \les p^2 M^{-1}
\label{hyp4a}
\end{align}

\noi
for $ s\geq 0$.
Here, the second inequality in \eqref{hyp4a} follows from the following estimate:
\begin{align*}
\sum_{\substack{\G_N \\ |n_1|>M}}
 \frac{1}
{\jb{n_1}^2 \jb{n_2}^2 \jb{n_3}^{2+\eps}\jb{n_4}^{2+\eps}}
& = \sum_{\substack{n_1, n_2, n_3 \in \Z^2\\|n_1| > M}}
 \frac{1}
{\jb{n_1}^2 \jb{n_2}^2 \jb{n_3}^{2+\eps}\jb{n_1+n_2 + n_3}^{2+\eps}}\\
& \les \sum_{\substack{n_1, n_2\in \Z^2\\|n_1| > M}}
 \frac{1}
{\jb{n_1}^2 \jb{n_2}^2 \jb{n_1+n_2}^{2+\eps} }\\
& \les \sum_{\substack{n_1\in \Z^2\\|n_1| > M}}
 \frac{1}
{\jb{n_1}^{4}}
\les M^{-2 }
\end{align*}

\noi
for any $\eps > 0$.
This proves~\eqref{McD3_bis} and hence~\eqref{McD3}.
This completes the proof of Lemma~\ref{LEM:NM}.
\end{proof}

Finally, we conclude this section by presenting the proof of Proposition~\ref{PROP:QFT}.

\begin{proof}[Proof of Proposition \ref{PROP:QFT}]
First, note that \eqref{parvo} follows from Lemma \ref{LEM:NM}.
Next, let us  show how Lemma~\ref{LEM:NM} implies \eqref{vtoro}.
It suffices to  show 
\begin{equation}
\label{integral}
\int_{1}^\infty 
\mu_{s,N, r}\big((u,v)\,:\, -F_{N}(u)>\log\lambda\big) \lambda^{p-1} d\lambda\leq C
\end{equation}

\noi
for some finite  $C>0$ {\it independent} of  the truncation parameter $N$. 
Here, $\mu_{s, N, r}$ is the Gaussian measure $\mu_s$ with 
a cutoff on the truncated energy $E_N(u, v)$
defined in \eqref{gauss4}.
While $F_N(u) = R_{s, N}(\pi_N u)$ is not sign-definite, 
the defocusing nature of the equation plays an important role.
In fact, from \eqref{K1} and  \eqref{K2} with \eqref{sigma}, we have
the following logarithmic bound:
\begin{align}
-F_N(u)\leq \frac 32 \s_N  \int_{\T^2} u^2
\leq C_r \log N, 
\label{hyp5}
\end{align}

\noi
in the support of $\mu_{s,N,  r}$.
In view of this logarithmic upper bound on $-F_N(u)$, 
we  apply Nelson's estimate \cite{Nelson2} to prove \eqref{integral}.
See \cite{DPT1, OTh} for analogous arguments in the context of 
the $\Phi_2^{2k}$-theory.

We need to estimate the measure 
\begin{align}
\mu_{s,N, r}\big((u,v): -F_{N}(u)>\log\lambda\big)
\label{hyp6}
\end{align}

\noi
for each given $\ld \geq 1$.
Choose $N_0 \in \R$ such that 
\begin{align*}
\log \lambda = 2 C_r  \log N_0. 
\end{align*}

\noi
Then, 
it follows from \eqref{hyp5} that 
the contribution to \eqref{hyp6} is 0 when $N < N_0$.
On the other hand, when $N \geq  N_0$, 
from \eqref{hyp5} and Lemma \ref{LEM:NM} (see Remark \ref{REM:hyp}), we have
\begin{align*}
\mu_{s,N, r}\big((u,v): -F_{N}(u)>\log\ld\big)
& \leq \mu_{s,N, r}\big((u,v): -F_{N}(u) + F_{N_0}(u) >\tfrac 12 \log\ld\big)\notag\\
&  \leq Ce^{-c (\log \ld)^\frac{1}{2}  N_0^{\frac{\theta}{2}}}
= Ce^{-c (\log \ld)^\frac{1}{2}  \ld^{\frac{\theta}{4C_r}}}.
\end{align*}

\noi
This exponential decay ensures the bound \eqref{integral} which in turn implies \eqref{vtoro}.

Finally, the uniform bound  \eqref{vtoro} implies \eqref{treto} by a standard argument (see \cite[Remark~3.8]{TZ2}). 
More precisely, the $L^p$-convergence  \eqref{treto} follows from the uniform $L^p$-bound \eqref{vtoro}
and the softer convergence in measure (as a consequence of \eqref{parvo}). 
This completes the proof of Proposition~\ref{PROP:QFT}.
\end{proof}

\begin{remark}\label{REM:3d}\rm
Let us briefly discuss the three-dimensional case.
By repeating the computation presented above, 
it is easy to check that Lemma \ref{LEM:NM} still holds
with $\theta = \min\big(\frac{1}{2}, s - 1\big)$, provided that $s > 1$.
The main issue in proving Proposition \ref{PROP:QFT} appears in \eqref{hyp5}.
In the three-dimensional case, 
we only have
\begin{align*}
-F_N(u)\leq C_r  N, 
\end{align*}

\noi
instead of the logarithmic bound \eqref{hyp5}.
If we were to repeat the argument above, this would force us to set
 $N_0 \in \R$ such that 
\begin{align*}
\log \lambda = 2 C_r   N_0, 
\end{align*}

\noi
leading to 
\begin{align}
\mu_{s,N, r}\big((u,v): -F_{N}(u)>\log\ld\big)
&  \leq Ce^{-c (\log \ld)^\frac{1}{2}  N_0^{\frac{\theta}{2}}}
= Ce^{-c' (\log \ld)^{\frac{1}{2}+ \frac{\theta}{2}}}.
\label{3d2}
\end{align}

\noi
Noting that $\theta = \frac 12$ when $s \geq \frac 32$, 
we see that \eqref{3d2} is not sufficient to guarantee \eqref{integral}. 
As in the construction of the $\Phi^4_3$-measure, 
one may need to introduce a further renormalization in the three-dimensional case.

Another modification appears in Lemma \ref{LEM:kin}.
In the three-dimensional case, 
the estimates~\eqref{kin0} and~\eqref{kin0a} hold
with $M^{\frac{1}{2}+\dl}$ (instead of $M^\dl$).
This loss 
makes the proof of Theorem \ref{THM:2}
presented in the next section
break down in the three-dimensional case.
For example,  in \eqref{Q2} below, 
we would have 
$ p N_{4}^{-1} N_1^{\frac 12 + \dl} N_2^{\frac 12 +\dl}$
(instead of $ p N_{4}^{-1} N_1^\dl N_2^\dl$),
which makes the computations in Case (ii) of Subsection \ref{SUBSEC:Q}
simply false in the three-dimensional case.
See Remark \ref{REM:3d2}.
\end{remark}

\section{Renormalized energy estimate}
\label{SEC:Q}

In this section, we establish the probabilistic energy estimate
on the renormalized energy (Theorem~\ref{THM:2}).
As in Subsection \ref{SUBSEC:1.4}, 
let $u_N = \pi_N u$ and $v_N = \pi_N v$.
Then, from \eqref{H8}, we have
$$
\dt E_{s, N}(\pi_N\Phi_N(t)(u,v))\vert_{t=0}=Q_1(u,v)+Q_2(u,v)+Q_3(u,v),
$$

\noi
where 
\begin{align}
Q_1(u,v)& = 3 \int_{\T^2}  \P_{\ne 0}[ (J^s u_N)^2  ] \, \P_{\ne 0} [v_N  u_N],
\label{Q0} \\
Q_2(u,v)& = 3 \bigg( \int_{\T^2} (J^s u_N)^2 - \s_N\bigg)\int_{\T^2}  v_Nu_N,  \notag \\
Q_3(u,v)& =
 \sum_{\substack{ |\al|+|\be|+|\g|\leq s\\
 |\al|,|\be|,|\g|<s}}
c_{\alpha,\beta,\gamma}
\int_{\T^2}
J^s  v_N\cdot\dd^\alpha u_N \cdot \partial^\beta u_N \cdot \partial^\gamma u_N,\notag
\end{align}

\noi
In the following, we prove
\begin{align}
\| Q_j(u,v)\|_{L^p(d\mu_{s, N, r})}\les p
\label{Q1}
\end{align}

\noi
for $j = 1, 2, 3$.

\subsection{Estimate on $Q_2(u,v)$}
By Cauchy-Schwarz and Cauchy's inequalities, we have
$$
\bigg|
\int_{\T^2}  v_N u_N
\bigg|
\leq \| u_N\|_{L^2}\|v_N\|_{L^2}
\leq E_N(u,v).
$$

\noi
Then, proceeding as in \eqref{hyp1b} with Lemma \ref{LEM:hyp}
and \eqref{hyp1c}, we have
\begin{align*}
\| Q_2(u,v)\|_{L^p(d\mu_{s, N, r})}
& \leq C_r\bigg\| \int_{\T^2} (J^s \pi_N u)^2 - \s_N\bigg\|_{L^p(d\mu_{s})}\\
& \sim \bigg\|\sum_{\substack{n \in \Z^2\\|n|\leq N}} \frac{|g_n(\omega)|^2 - 1}{\jb{n}^2} \bigg\|_{L^p(\O)}
 \leq p \, \bigg\|\sum_{\substack{n \in \Z^2\\|n|\leq N}} \frac{|g_n(\omega)|^2 - 1}{\jb{n}^2} \bigg\|_{L^2(\O)}\\
& \les p.
\end{align*}

\noi
This proves \eqref{Q1} in this case.

\subsection{Estimate on $Q_1(u,v)$}\label{SUBSEC:Q}
By applying the  Littlewood-Paley decomposition, we have
$$
Q_1(u,v)= \sum_{\substack{ N_1,N_2,N_3,N_4\geq 1 \\\text{ dyadic}}}
Q_1^{\bf N}(u,v),
$$
where 
${\bf N}:=( N_1,N_2,N_3,N_4)$ and 
\begin{align}
Q_1^{\bf N}(u,v)=
3 \int_{\T^2}  \P_{\ne 0}[ J^s \P_{N_1}u_N \cdot J^s \P_{N_2}u_N  ] \, \P_{\ne 0} [\P_{N_3} v_N \cdot \P_{N_4} u_N ].
\label{Q1a}
\end{align}

\noi
We consider several cases according to the sizes of $N_1$, $N_2$, $N_3$, $N_4$.

\medskip

\noi
{\bf Case (i):} $N_4 \ges \max( N_1, N_2)^\frac{1}{100}$.
\\
\indent
Since $\P_{\ne 0}$ is clearly bounded on $L^p(\T^2)$, $1\leq p\leq \infty$, we have 
\begin{align*}
|Q_1^{{\bf N}}( u, v)|
& \les 
\big\|\P_{\ne 0}[ J^s \P_{N_1} u_N\cdot  J^s \P_{N_2} u_N ]\big\|_{L^{\infty}_x}
\big\| \P_{\ne 0} [\P_{N_3} v_N \cdot \P_{N_4} u_N ]\big\|_{L^1_x}\\
& \leq \|J^s \P_{N_1} u_N\|_{L^\infty_x}
\|J^s \P_{N_2}u_N\|_{L^\infty_x}
\| \P_{N_3} v_N\|_{L^2_x}
\| \P_{N_4} u_N\|_{L^2_x}.
\end{align*}

\noi
Noting that 
$$
\| \P_{N_3} v_N\|_{L^2_x}
\| \P_{N_4} u_N\|_{L^2_x}\les N_{4}^{-1} E_N(u,v),
$$

\noi
we have
$$
\| Q_1^{{\bf N}}(u,v)\|_{L^p(d\mu_{s, N, r})}
\leq C_r N_{4}^{-1}
\Big\|
\|J^s \P_{N_1} u_N\|_{L^\infty_x}
\|J^s \P_{N_2}u_N\|_{L^\infty_x}
\Big\|_{L^p(d\mu_{s})}.
$$

\noi
Thanks to Lemma~\ref{LEM:kin}, we have 
$$
\big\|\|J^s \P_{N_j} u_N\|_{L^\infty_x}\big\|_{L^{2p}(\mu_{s})}\leq C_\dl\,  \sqrt p N_{j}^\delta
$$

\noi
for any  $\delta>0$, $j = 1, 2$.
Hence,  for any $\delta>0$, we have
\begin{align}
\| Q_1^{{\bf N}}(u,v)\|_{L^p(d\mu_{s, N, r})}
& \leq C_r N_{4}^{-1}
\Big\|
\|J^s \P_{N_1} u_N\|_{L^\infty_x}
\|J^s \P_{N_2}u_N\|_{L^\infty_x}
\Big\|_{L^p(d\mu_{s})}\notag \\
& \les  p N_{4}^{-1} N_1^\dl N_2^\dl.
\label{Q2}
\end{align}

\noi
By noting that $Q_1^{\bf N}(u, v)$ is not trivial only if 
$$
N_{3}\lesssim N_1+N_2+N_4, 
$$

\noi
 we can readily sum \eqref{Q2} over the dyadic blocks $N_j$, $j = 1, \dots, 4$.
This yields \eqref{Q1} in this case.

 \begin{remark}\label{REM:3d2}\rm
Thanks to Case (i), we can restrict the range of $N_4$ in the following.
This restriction: $N_4 \ll \max( N_1, N_2)^\frac{1}{100}$ plays a crucial role in Case (ii) presented below.
In the three-dimensional case,  due to the weaker conclusion of Lemma \ref{LEM:kin}
mentioned in Remark \ref{REM:3d}, 
we would have 
$ p N_{4}^{-1} N_1^{\frac 12 + \dl} N_2^{\frac 12 +\dl}$
on the right-hand side of \eqref{Q2}.
In particular, the argument above allows us to  conclude \eqref{Q1} 
under a much stronger  condition: 
 $N_{4}\ges N_1^{\frac 12 + 2\dl} N_2^{\frac 12 +2\dl}$,
 preventing us to handle the remaining case: 
 $N_{4}\ll N_1^{\frac 12 + 2\dl} N_2^{\frac 12 +2\dl}$
 in the three-dimensional setting.

 \end{remark}

\noi
{\bf Case (ii):} $N_4 \ll \max( N_1, N_2)^\frac{1}{100}$.
\\
\indent
In this case, we have $\max( N_1, N_2) \sim \max \{N_j, j = 1, \dots, 4\}$.
Without loss of generality, assume $N_2 \leq N_1 \sim \max \{N_j, j = 1, \dots, 4\}$.
Let  $a= a(s) > 0$ be sufficiently small (to be chosen later).
We consider the following two cases:
\[
 \text{(ii.a): }\ N_3 \ll N_1^{1-a}
\qquad \text{and} \qquad 
\text{(ii.b):}\ N_3 \ges N_1^{1-a}.
\]

\noi
$\bullet$ {\bf Subcase (ii.a):} $N_3 \ll N_1^{1-a}$.
\\
\indent
In this case, we have  $N_1 \sim N_2$. 
By H\"older's inequality, we have 
\begin{align}
\|  \ft u_{n} \|_{\l^{1+}_{n}}\les  E_N(u,v)^\frac{1}{2}.
\label{Q3}
\end{align}

\noi
Then, given $p \geq 2$, it follows from Young's inequality, \eqref{Q3}, 
and Minkowski's inequality that
\begin{align*}
\big\|     Q_1^{{\bf N}} &  ( u,v )\big\|_{L^p(d\mu_{s, N, r})}
 \\
& \les \Bigg\|
\bigg\|  
\sum_{\substack{n = n_1 + n_2
\\|n_j|\sim N_j, |n_j|\leq N,  \, j=1,2
\\
1\leq |n_1 + n_2| \ll N_1^{1-a }}
}
\jb{n_1}^{s}\ft{ u}_{n_1}
\jb{n_2}^{s}\ft{ u}_{n_2}
\bigg\|_{\l^{2-}_n}
 \underbrace{\| \ft {v}_{n_3} \|_{\l^2_{n_3}} \| \ft u_{n_4} \|_{\l^{1+}_{n_4}} }_{\les E_N(u,v)}
\Bigg\|_{L^p(d\mu_{s, N, r})}
\\[2mm]
&  \leq C_{r,\eps} 
\Bigg\|\bigg\|  
\sum_{\substack{n = n_1 + n_2
\\
|n_j|\sim N_j, |n_j|\leq N,  \, j=1,2
\\
1\leq |n_1 + n_2| \ll N_1^{1-a }}
} 
\jb{n_1}^{s}\ft{ u}_{n_1}
\jb{n_2}^{s}\ft{ u}_{n_2}
\bigg\|_{\l^{2-\eps}_n}
\Bigg\|_{L^p(d\mu_s)}\\[2mm]
& \leq C_{r,\eps} 
\Bigg\|\bigg\|  
\sum_{\substack{n = n_1 + n_2
\\
|n_j|\sim N_j, |n_j|\leq N, \,  j=1,2
\\
1\leq |n_1 + n_2| \ll N_1^{1-a }}
} 
\jb{n_1}^{s}\ft{ u}_{n_1}
\jb{n_2}^{s}\ft{ u}_{n_2}
\bigg\|_{L^p(d\mu_s)}
\Bigg\|_{\l^{2-\eps}(|n|\lesssim N_1^{1-a})}
\end{align*}

\noi
for any small $\eps > 0$.
Here, we have $n_1 + n_2 \ne 0$ thanks to the first projection $\P_{\ne 0}$
in the definition \eqref{Q0} of $Q_1(u, v)$, 
while we have 
$ |n_1 + n_2| = |n_3 + n_4| \les \max(N_3, N_4) \ll N_1^{1-a }$.
By the Wiener chaos estimate (Lemma \ref{LEM:hyp}) with \eqref{series} and $N_1 \sim N_2$, 
we have
\begin{align*}
\big\| Q_1^{{\bf N}} (u,v )\big\|_{L^p(d\mu_{s, N, r})}
& \leq C_{r,\eps} \,
 p \,
\Bigg\|\bigg\|\sum_{\substack{n = n_1 + n_2
\\
|n_j|\sim N_j, |n_j|\leq N,  \, j=1,2
\\
1\leq |n_1 + n_2| \ll N_1^{1-a }}
}
 \frac{g_{n_1}(\omega)}{\jb{n_1}}
\frac{g_{n_2 }(\omega)}{\jb{n_2}} \bigg\|_{L^2(\O)}
\Bigg\|_{\l^{2-\eps}_{n}(|n| \les N_1^{1-a})} 
\\[2mm]
& \lesssim
C_{r,\eps} \,
 p \,
 \bigg\|\Big(\sum_{|n_1|\sim N_1}
N_1^{-4}\Big)^\frac{1}{2}
\bigg\|_{\l^{2-\eps}_{n}(|n| \les N_1^{1-a})} 
\\
& \sim p N_1^{-1}
\big\|\ind_{|n| \les N_1^{1-a}}\big\|_{\l^{2-\eps}_{n}} 
 \sim p N_1^{-1} 
N_1^{\frac{2-2a}{2 - \eps} }
= p  
N_1^{\frac{-2a+\eps}{2 - \eps} }.
\end{align*}

\noi
Therefore, 
by choosing sufficiently small $\eps>0$ such that $\eps<2a$, 
we have a negative power of $N_1$ that can be used to sum over the dyadic blocks.
This proves \eqref{Q1} in this case.

\medskip

\noi
$\bullet$ {\bf Subcase (ii.b):} $N_3\ges N_1^{1-a}$.
\\
\indent
By  Young's inequality, \eqref{Q3},
and H\"older's inequality, we have
\begin{align*}
\big\|   Q_1^{\bf N}  (u, v )\big\|_{L^p(d\mu_{s, N, r})} 
& \les
\Bigg\|  \bigg\|  
\sum_{\substack{n = n_1 + n_2+n_3 
\\
|n_j|\sim N_j,  |n_j|\leq N, \, j=1,2,3
\\n_{1}+n_2\ne0
\\ |n_1+n_2+n_3| \ll N_1^\frac{1}{100} }}
\jb{n_1}^{s} \ft{u}_{n_1}\jb{n_2}^{s}\ft u_{n_2}  \ft{v}_{n_3}
\bigg\|_{\l^{q}_n}
 \underbrace{  \|  \ft u_{n_4} \|_{\l^{1+}_{n_4}}}_{ \les E_N(u,v)^\frac{1}{2}}
\Bigg\|_{L^p(d\mu_{s, N, r})}\\[2mm]
& \les
C_r \Bigg\|  \ind_{\{E_N(u,v) \leq r\}}\bigg\|  
\sum_{\substack{n = n_1 + n_2+n_3 
\\
|n_j|\sim N_j,  |n_j|\leq N, \, j=1,2,3
\\n_{1}+n_2\ne0
\\ |n_1+n_2+n_3| \ll N_1^\frac{1}{100} 
}
}
\jb{n_1}^{s} \ft{u}_{n_1}\jb{n_2}^{s}\ft u_{n_2} \ft{v}_{n_3}
\bigg\|_{\l^{q}_n}
\Bigg\|_{L^p(d\mu_s)}
\end{align*}

\noi
for some $q \gg 1$ (to be chosen later).
Now, we can trivially write
\begin{align*}
\big\| Q_1^{\bf N}   (u,  v )\big\|_{L^p(d\mu_{s,N, r})}
&  \les
\Bigg\|  \ind_{\{E_N(u,v) \leq r\}}
\Bigg(
\sum_{|n| \ll N_1^{\frac{1}{100}}}
\bigg|
\sum_{\substack{n = n_1 + n_2+n_3 \\n_{1}+n_2\ne0\\|n_j|\sim N_j, |n_j|\leq N, \,  j=1,2,3 }}
\jb{n_1}^{s} \ft{u}_{n_1}
\jb{n_2}^{s}\ft u_{n_2}  \ft{v}_{n_3}
\bigg|^{\frac{q}{3}}
\\
& \hphantom{XXXX}
\times \bigg|\sum_{\substack{n = n_1 + n_2+n_3 \\n_{1}+n_2\ne0
\\|n_j|\sim N_j, |n_j|\leq N, \, j=1,2,3 }} 
\jb{n_1}^{s} \ft{u}_{n_1}
\jb{n_2}^{s}\ft u_{n_2}  \ft{v}_{n_3}
\bigg|^{\frac{2q}{3}}
\Bigg)^\frac{1}{q} \Bigg\|_{L^p(d\mu_{s})}.
\end{align*}

\noi
In the following, we estimate  the first and  second factors on the right-hand side above in a different manner. 
For the first factor, we shall use the energy restriction $E_N(u,v) \leq r$, 
while, for the second factor, we shall invoke the Wiener chaos estimate (Lemma \ref{LEM:hyp}). 
The balance between the powers is chosen so that we obtain $p$ to power one at the end. 
The main point  in this procedure is that we get tractable bounds with respect 
to the dyadic frequency  localization. 
Consequently,  in the case under consideration, we have
\begin{align}
\big\|Q_1^{\bf N}  & ( u, v   )\big\|_{L^p(d\mu_{s, N, r})} \notag\\
&  \les
\Bigg\| \ind_{\{E_N(u,v) \leq r\}}
\Bigg(
\sum_{|n| \ll N_1^{\frac{1}{100}}}
\bigg(
N_1^{2s-2}
\|\jb{n_1}  \ft{u}_{n_1}\|_{\l^2_{n_1}}
\|\jb{n_2} \ft{u}_{n_2}\|_{\l^2_{n_2}}
\underbrace{\|\ft{v}_{n_3}\|_{\l^1_{n_3}}}_{\les N_1\| v\|_{L^2}}
\bigg)^{\frac{q}{3}}
\notag \\
& \hphantom{XXXX}\times \bigg|
\sum_{
\substack{n = n_1 + n_2+n_3 
\\n_{1}+n_2\ne0
\\
|n_j|\sim N_j, |n_j|\leq N, j=1,2,3
}
}
\jb{n_1}^{s} \ft{u}_{n_1}
 \jb{n_2}^{s}\ft u_{n_2}  \ft{v}_{n_3}
\bigg|^{\frac{2q}{3}}
\Bigg)^\frac{1}{q} \Bigg\|_{L^p(d\mu_{s})}
\label{Q4}
\\
&  \leq C_r 
N_1^{\frac{2s-1}{3}}
\Bigg\|
\Bigg(\sum_{|n| \ll N_1^{\frac{1}{100}}}
\bigg|
\sum_{\substack{n = n_1 + n_2+n_3\\n_{1}+n_2\ne0 \\|n_j|\sim N_j, |n_j|\leq N, j=1,2,3 }} 
\jb{n_1}^{s} \ft{u}_{n_1}
\jb{n_2}^{s}\ft u_{n_2} \ft{v}_{n_3}
\bigg|^{\frac{2q}{3}}
\Bigg)^\frac{3}{2q} \Bigg\|_{L^{\frac{2p}{3}}(d\mu_s)}^{\frac{2}{3}}.
\notag 
\end{align}

Without loss of generality, assume  $p\geq q$.
Then, by Minkowski's inequality and  the Wiener chaos estimate (Lemma \ref{LEM:hyp})
with \eqref{series}, 
we have 
\begin{align}
\Bigg\|
\Bigg(\sum_{|n| \ll N_1^{\frac{1}{100}}}
& \bigg|
\sum_{\substack{n = n_1 + n_2+n_3\\n_{1}+n_2\ne0 \\|n_j|\sim N_j, |n_j|\leq N, \, j=1,2,3 }} 
\jb{n_1}^{s} \ft{u}_{n_1}
\jb{n_2}^{s}\ft u_{n_2}  \ft{v}_{n_3}
\bigg|^{\frac{2q}{3}}
\Bigg)^\frac{3}{2q} \Bigg\|_{L^{\frac{2p}{3}}(d\mu_s)}
\notag \\[2mm]
& \leq
\Bigg\|
\bigg\|
\sum_{\substack{n = n_1 + n_2+n_3\\n_{1}+n_2\ne0 \\|n_j|\sim N_j, |n_j|\leq N, \, j=1,2,3 }} 
\jb{n_1}^{s} \ft{u}_{n_1}
\jb{n_2}^{s}\ft u_{n_2}  \ft{v}_{n_3}
\bigg\|_{L^{\frac{2p}{3}}(d\mu_s)}
\Bigg\|_{\l^{\frac{2q}{3}}(|n|\ll N_1^{\frac{1}{100} })}
\notag  \\[2mm]
& \le
p^{\frac{3}{2}}
\Bigg\|
\bigg\|
\sum_{\substack{n = n_1 + n_2+n_3\\n_{1}+n_2\ne0 \\|n_j|\sim N_j, |n_j|\leq N, \, j=1,2,3 }} 
\jb{n_1}^{s} \ft{u}_{n_1}
\jb{n_2}^{s}\ft u_{n_2}  \ft{v}_{n_3}
\bigg\|_{L^2(d\mu_s)}
\Bigg\|_{\l^{\frac{2q}{3}}(|n|\ll N_1^{\frac{1}{100} })}
\notag  \\[2mm]
& =p^{\frac{3}{2}}
\Bigg\|
\bigg\|
\sum_{\substack{n = n_1 + n_2+n_3\\n_{1}+n_2\ne0 \\|n_j|\sim N_j, |n_j|\leq N, \,j=1,2,3 }} 
\frac{g_{n_1}}{\jb{n_1}}
\frac{g_{n_2}}{\jb{n_2}}
\frac{h_{n_3}}{\langle n_3\rangle^{s}}
\bigg\|_{L^2(\Omega)}
\Bigg\|_{\l^{\frac{2q}{3}}(|n|\ll N_1^{\frac{1}{100} })}
\notag  \\[2mm]
 &  \lesssim p^{\frac 32} 
\Bigg\|\bigg(
\sum_{|n_j|\sim N_j, \, j=2,3 } 
N_1^{-2} \frac{1}{\langle n_2 \rangle^2}\frac{1}{\langle n_3\rangle^{2s}} \bigg)^\frac{1}{2}
\Bigg\|_{\l^{\frac{2q}{3}}_{n}(|n| \ll N_1^{\frac 1{100}})}.
\label{Q5}
\end{align}

\noi
Summing over $n_2$ and $n_3$ with  $|n_2| \sim N_2$ and 
$ N_1^{1-a}\les N_3 \les N_1$, 
we have
\begin{align}
\text{LHS of }\eqref{Q5} 
& \les
 p^{\frac 32} \,
\bigg\|\Big(
N_1^{-2} N_1^{(-2s+2)(1-a)}
\Big)^\frac{1}{2}
\bigg\|_{\l^{\frac{2q}{3}}_{n}(|n| \ll N_1^{\frac 1{100}})}
\notag \\
&  =  p^\frac{3}{2} N_1^{-s + as - a}
\big\|\ind_{|n| \ll N_1^{\frac 1{100}}}
\big\|_{\l^{\frac{2q}{3}}_{n}}
 \les p^\frac{3}{2}
 N_1^{-s + as - a}N_1^{\frac{3}{100 q} }.
\label{Q5a}
\end{align}

\noi

Therefore, by choosing sufficiently large $q\gg 1$ and sufficiently small $a = a(s) > 0$, it follows
from \eqref{Q4} and \eqref{Q5a} that 
$$
\big\|Q_1^{\bf N}  ( u, v )\big\|_{L^p(d\mu_{s,N, r})}
\les Cp N_{1}^{-\frac13 +\frac 23 as - \frac 23 a+ \frac 1{50q}}
\les Cp N_{1}^{-\al}
$$

\noi
for some $\al > 0$.
Once again, we obtained a negative power of $N_1$,
allowing us to  sum over the dyadic blocks.
This proves \eqref{Q1} in  Subcase (ii.b).


\subsection{Estimate on $Q_3(u,v)$}

It remains to prove \eqref{Q1} for $j = 3$.
It turns out that  $Q_3(u,v)$ can be estimated  essentially in  the same manner as  $Q_1(u,v)$.
By integration by parts, we can express each summand in the definition of 
 $Q_3(u,v)$ as
\begin{equation}
\label{Q7}
\int_{\T^2}
\dd^{\kk}  v_N\cdot \dd^\al u_N \cdot\dd^\be u_N \cdot \dd^\g u_N,
\end{equation}

\noi
where $|\kk|\leq s-1$, $|\al|+|\be|+|\g| \leq s+1$,  and $\max(|\alpha|,|\beta|,|\gamma|)\leq s$. 

Let us first consider the case $\max(|\alpha|,|\beta|,|\gamma|)= s$.
By symmetry, we assume  that $|\alpha|\geq |\beta|\geq |\gamma|$ and therefore $|\alpha|=s$. We then necessarily have $|\beta|=1$ and $|\gamma|=0$. 
Then, 
we can treat \eqref{Q7} exactly in the same manner as we did for $Q_1(u,v)$ 
by  replacing
$ \P_{\ne 0}[ (J^s u_N)^2  ]$
and $\P_{\ne 0} [v_N  u_N]$ in the definition \eqref{Q0} of $Q_1(u, v)$
with $\dd^\kk v_N\cdot \dd^\al u_N$
and $\dd^\be u_N \cdot \dd^\g u_N$, respectively.
Note that, while  
the frequency projection 
$ \P_{\ne 0}[ (J^s u_N)^2  ]$
in the definition of $Q_1(u, v)$ played an important role
in eliminating the logarithmic divergence, 
we do not need a frequency projection $\P_{\ne 0}$ on 
$\dd^\kk v_N\cdot \dd^\al u_N$
since, in view of \eqref{series}, the independence of $v_N$ and $u_N$ prevents
such logarithmic divergence.

Therefore, we can suppose that  
$\max(|\alpha|,|\beta|,|\gamma|)\leq  s-1$.
We only consider the worst case $|\al|+|\be|+|\g| = s + 1$ and $|\kk| = s-1$ in the following.
In this case, noting that 
$\dd^\kk v_N $ 
with $|\kk| = s-1$
behaves  like $J^s \pi_N u$ (see \eqref{series}), 
we can basically proceed as we did for $Q_1(u, v)$ in the previous subsection.
 Indeed, by applying the Littlewood-Paley decomposition, we need 
to study  the expression of the form
\begin{equation*}
\wt Q_1^{\bf N}(u, v) = \int_{\T^2}
\dd^\kk  \P_{N_1} v_N\cdot \dd^\al \P_{N_2}u_N\cdot 
\dd^\be \P_{N_3} u_N\cdot \dd^\g \P_{N_4}u_N.
\end{equation*}

\noi
By symmetry, assume $N_2 \geq N_3 \geq N_4$.
Then, we have
\begin{equation}\label{Q8}
\wt Q_1^{\bf N}(u, v) \sim  \int_{\T^2}
\dd^\kk  \P_{N_1} v_N\cdot 
N_2^{s - |\al|} \dd^\al \P_{N_2}u_N\cdot 
N_2^{1 - |\be|}\dd^\be \P_{N_3} u_N\cdot 
N_2^{-|\g|}\dd^\g \P_{N_4}u_N.
\end{equation}

\noi
As mentioned above,  the first factor
$\dd^\kk  \P_{N_1} v_N$
 in \eqref{Q8} behaves like 
 $J^s \P_{N_1}u_N$ in \eqref{Q1a}.
 The second factor
$N_2^{s - |\al|} \dd^\al \P_{N_2}u_N 
\sim \dd^{\wt \al} \P_{N_2} u_N$ with $|\wt \al| = s$
 also behaves like  the second factor  $J^s \P_{N_2}u_N$ in \eqref{Q1a}.
Similarly, the third and fourth factors in \eqref{Q8}:
\[N_2^{1 - |\be|}\dd^\be \P_{N_3} u_N 
\, \text{``}\!\les \!\text{''}\, 
\nb \P_{N_3} u_N 
\qquad \text{and}\qquad
N_2^{-|\g|}\dd^\g \P_{N_4}u_N
\, \text{``}\!\les \!\text{''}\, 
\P_{N_4}u_N\]

\noi
behave (at worst) like the third and fourth factors in \eqref{Q1a}, respectively.
Hence, we can estimate  $\wt Q_1^{\bf N}(u, v)$
just as we did for $ Q_1^{\bf N}(u, v) $ in the previous section.
This completes the proof of Theorem \ref{THM:2}.


\section{Proof of Theorem~\ref{THM:NLKG}}\label{SEC:proof_th1} 

In this section, we prove
quasi-invariance of the Gaussian measure $\mu_s$
under the NLKG dynamics (Theorem \ref{THM:NLKG}).
While the general structure of the argument is similar to our previous works \cite{Tzvet, OTz}
(see also~\cite{OTz2} for a concise sketch of the general structure), 
we proceed differently in some  part (see Proposition~\ref{PROP:meas3}).

\subsection{A change-of-variable formula}
As in  our previous works \cite{Tzvet, OTz} ,
the change-of-variable formula (Lemma~\ref{LEM:cov}) 
for the nonlinear transformation induced by the truncated flow $\Phi_N(t)$
plays an  important role.
We also point out that these change-of-variable formulas
in this paper and in \cite{Tzvet, OTz}
are  in turn  inspired by~\cite{TzV1}.

Let $\Ld$ be as in \eqref{Ld1}.
Given  $N \in \N$, we denote by $\EE_N$ the real vector space: 
$$
\mathcal{E}_N = \text{span} \big\{1, \cos(n\cdot x),\, \sin(n\cdot x): n \in \Ld_N^*\big\}, 
$$

\noi
where $\Ld_N^* = \{n \in \Z^2:  0 < |n| \leq N \} \cap \Ld$.
We equip $\EE_N$ with the natural scalar product.
Moreover, we endow $\EE_N\times \EE_N$ with a Lebesgue measure $L_N$ as follows.
Given
$$
(\pi_Nu)(x)=\sum_{|n|\leq N} \ft u_n \, e^{in\cdot x},\qquad \ft u_{-n}=\overline{\ft u_{n}}, 
$$

\noi
let  $a_n = \Re\ft u_n$ and $b_n = \Im \ft u_n$, $(a_n,b_n)\in \R^2$.
Then, we have
$$
(\pi_Nu)(x)=a_0+ \sum_{n \in \Ld_N^*}\big\{a_n (2\cos(n\cdot x))+b_n (-2\sin(n\cdot x))\big\}.
$$

\noi
Therefore, it is natural to define $L_N$ as the Lebesgue measure on $\EE_N\times \EE_N$ 
 with respect to the orthogonal basis: 
$$
\Big\{1, \big\{2\cos(n\cdot x),-2\sin(n\cdot x)\big\}_{n \in \Ld_N^*}\Big\}
\times \Big\{1, \big\{2\cos(n\cdot x),-2\sin(n\cdot x)\big\}_{n \in \Ld_N^*}\Big\}.
$$

Next, we denote by $(\EE_N\times \EE_N)^\perp$ 
the orthogonal complement of $\EE_N\times \EE_N$ in  $\H^\s(\T^2)$, $\sigma<s$.
We endow $(\EE_N\times \EE_N)^\perp$ with the marginal Gaussian  measure $\mu^\perp_{s;N}$ 
on $\pi_N^\perp \H^\s(\T^2)$ 
which is defined as the induced probability measure under the map:
$$
\o \in \O \longmapsto (\pi_N^\perp u^\o,\pi_N^\perp v^\o),
$$

\noi
where $(u^\o, v^\o)$ is as in \eqref{series}.
By viewing  the Gaussian measure  $\mu_s$ as a product measure on $(\EE_N\times \EE_N)
\times (\EE_N\times \EE_N)^\perp$, 
we can write the truncated weighted Gaussian measure 
$\rho_{s, N, r}$ defined in \eqref{K0} as
\begin{align*}
d\rho_{s, N, r}(u,v)&  =  Z_{s,N, r}^{-1} \ind_{\{E_N(u,v) \leq r\}}\,
e^{-  R_{s, N}(\pi_{N}u)} d \mu_s(u,v) 
\\
& =  \hat{Z}_{s,N, r}^{-1} \ind_{\{E_N(u,v) \leq r\}}\,
e^{-  E_{s, N}(\pi_{N}u,\pi_N v)}\,  dL_N\otimes d  \mu^\perp_{s;N},
\end{align*}

\noi
where $ \hat{Z}_{s,N, r}$ is defined by
$$
 \hat{Z}_{s,N, r}
 =\int_{ {\mathcal H}^\sigma(\T^2)} \ind_{\{E_N(u,v) \leq r\}}\,
 e^{-  E_{s, N}(\pi_{N}u,\pi_N v)}\,  dL_N\otimes  d\mu^\perp_{s;N}.
 $$
 
 \noi
 Then, we have the following change-of-variable formula.
 
\begin{lemma}\label{LEM:cov}
Let $s>1$, $N \in \N$, and $r > 0$.
Then, we have 
$$
\rho_{s,N,r}(\Phi_N(t)(A))=\hat{Z}_{s,N, r}^{-1} 
\int_A
\ind_{\{E_N(u,v) \leq r\}}\,
e^{-  E_{s, N}(\pi_{N}\Phi_N(t)(u,v))}
 \, dL_N\otimes d \mu^\perp_{s;N}
$$

\noi
for any $t \in \R$ and any measurable set
 $A\subset {\mathcal H}^\sigma(\T^2)$, $\sigma\in (1,s)$.
\end{lemma}

Lemma~\ref{LEM:cov} 
follows from similar considerations presented in~\cite{Tzvet, OTz} 
and therefore we omit its proof.

\subsection{The evolution of the truncated measures}

We now study the evolution of the truncated measures $\rho_{s, N, r}$.
We shall use the renormalized energy estimate
(Theorem~\ref{THM:2})   as a key step in the proof of the following statement. 
Due to the use of Theorem~\ref{THM:2}, 
 we assume that $s\geq 2$ is an even integer in the following.
While all the implicit constants depend on $s$, we may not state their dependence
in an explicit manner.

\begin{lemma}\label{LEM:meas2}
 Given  $r > 0$, 
there exists $C_{ r}>0$ such that 
\begin{align*}
\frac{d}{dt} \rho_{s, N, r}(\Phi_N(t) (A))
\leq C_{r} p \, \big\{ \rho_{s, N, r} (\Phi_N(t)(A))\big\}^{1-\frac 1p}
\end{align*}
\noi
for any $p \geq 2$, 
any $N \in \N$, any $t \in \R$, 
and  any measurable set 
$A \subset \H^\sigma(\T^2)$, $\sigma\in (1,s)$.
\end{lemma}

While the proof of Lemma~\ref{LEM:meas2} also follows
from the argument in our previous works~\cite{Tzvet, OTz}, 
we present its details in order to show the  use 
of the crucial renormalized energy estimate.

\begin{proof}
Fix $t_0\in\R$. 
As in~\cite{TzV1, Tzvet, OTz}, the main idea is to reduce
the analysis to that at $t = 0$.
Using the flow property of $\Phi_N(t)$, we have
\begin{align*}
\frac{d}{dt} \rho_{s, N, r}(\Phi_N(t) (A))\bigg|_{t=t_0}
& =
 Z_{s,N, r}^{-1} 
\frac{d}{dt}
\int_{\Phi_N(t)(A)}\ind_{\{E_N(u,v) \leq r\}}\,e^{-  R_{s, N}(\pi_{N}u)} d \mu_s(u,v)\bigg|_{t=t_0}
\\
& =
Z_{s,N, r}^{-1} 
\frac{d}{dt}
\int_{\Phi_N(t)(\Phi_{N}(t_0)(A))}\ind_{\{E_N(u,v) \leq r\}}\,e^{-  R_{s, N}(\pi_{N}u)} d \mu_s(u,v)\bigg|_{t=0}.
\end{align*}

\noi
By  the change-of-variable formula (Lemma~\ref{LEM:cov}),  we have
\begin{align*}
\frac{d}{dt} & \rho_{s, N, r}  (\Phi_N  (t) (A))\bigg|_{t=t_0}
\\
& =
\hat{Z}_{s,N, r}^{-1} 
\frac{d}{dt}
\int_{\Phi_{N}(t_0)(A)}
\ind_{\{E_N(u,v) \leq r\}}\,
e^{-  E_{s, N}(\pi_{N}\Phi_N(t)(u,v))}
  dL_N\otimes  d \mu^\perp_{s;N}
\bigg|_{t=0}
\\
& =Z_{s,N, r}^{-1} 
\int_{\Phi_{N}(t_0)(A)}\ind_{\{E_N(u,v) \leq r\}}\,
\dt E_{s, N}(\pi_N\Phi_N(t)(u,v)) |_{t=0}\,
e^{-  R_{s, N}(\pi_{N}u)} d \mu_s(u,v) .
\end{align*}

\noi
Now, H\"older's inequality yields
\begin{align*}
\frac{d}{dt} \rho_{s, N, r}(\Phi_N(t) (A))\bigg|_{t=t_0}
& \leq 
\big\|
\dt E_{s, N}(\pi_N\Phi_N(t)(u,v))|_{t=0}
\big\|_{L^p(\rho_{s, N, r})} \\
& \hphantom{XXXXXXXXX}
\times
\big\{\rho_{s, N, r}(\Phi_N(t_0) (A))\big\}^{1-\frac{1}{p}}.
\end{align*}
Observe that  Proposition~\ref{PROP:QFT} implies that $Z_{s,N, r}^{-1}$ is bounded, uniformly in $N$. 
Finally, by Cauchy-Schwarz inequality
 together with the uniform estimate \eqref{vtoro} in Proposition~\ref{PROP:QFT} 
 and Theorem~\ref{THM:2}, we obtain
 \begin{align*}
\big\|\dt E_{s, N}
& (\pi_N\Phi_N(t)(u,v))|_{t=0}
\big\|_{L^p(\rho_{s, N, r})} \\
& \leq Z_{s,N, r}^{-\frac{1}{p}}\,
\Big\|
\dt E_{s, N}(\pi_N\Phi_N(t)(u,v))\vert_{t=0}
\Big\|_{L^{2p}(\mu_{s, N, r})}\,
 \Big\|\ind_{\{E_N(u,v) \leq r\}}\,e^{-  R_{s, N}(\pi_{N}u)}  \Big\|_{L^{2}(\mu_s)}^\frac{1}{p}\\
&  \leq C_r p,
 \end{align*}
 
 \noi
since  
 $ Z_{s,N, r}^{-\frac{1}{p}} \leq C(s, r)$ for any $p \geq 2$ and $N \in \N$. 
This completes the proof of Lemma~\ref{LEM:meas2}.
\end{proof}

As a corollary to Lemma~\ref{LEM:meas2},  
we obtain the following control on the truncated measures $\rho_{s, N, r}$. 
We point out that this is where our argument diverges from the presentation in our previous works~\cite{Tzvet, OTz}.

\begin{proposition}\label{PROP:meas3}
Given $r > 0$, there exists $t_{r} > 0$ such that given $\eps > 0$,  there exists $\dl > 0$ such that
 if,  for a measurable set $A \subset {\mathcal H}^\sigma(\T^2)$, $\sigma\in (1,s)$,
 there exists $N_0 \in \N$ such that 
$$
\rho_{s, N, r} ( A)< \dl
$$

\noi
for any $N \geq N_0$, then  we have 
$$
\rho_{s, N, r} (\Phi_N(t) (A)) < \eps
$$

\noi
for any $t \in [0, t_{r}]$ and any $N \geq N_0$.

\end{proposition}

\begin{remark}\rm

In Proposition~\ref{PROP:meas3}, 
we can choose  $t_r > 0$ and $\dl > 0$ such that they are independent of $N \in \N$.
Moreover,  $\dl > 0$ is independent of $t_r>0$.

\end{remark}

\begin{proof}
From  Lemma~\ref{LEM:meas2}, we have 
\begin{align}
\frac{d}{dt}\big\{ \rho_{s, N, r}(\Phi_N(t) (A))\big\}^\frac{1}{p}
\leq C_{r}
\label{meas3}
\end{align}

\noi
for any $p \geq 2$.
Integrating \eqref{meas3} from 0 to $t$, we obtain
\begin{align}
 \rho_{s, N, r}(\Phi_N(t) (A))
&  \leq 
\big\{ \big(  \rho_{s, N, r}(A)\big)^\frac{1}{p} + C_{ r} t\big\}^p.
\label{meas4}
\end{align}

\noi
Now, choose $t_r > 0$ such that $C_{r}t_r = \frac 14$.
Without loss of generality,  assume  $\eps<1$. 
It follows from \eqref{meas4} and  the convexity inequality:
$$
\Big(
\frac{x+y}{2}
\Big)^p
\leq\frac{ x^p+y^p}{2},\qquad x,y\geq 0,\, p\geq 1
$$
that for $t\in [0,t_r]$,
\begin{align*}
 \rho_{s, N, r}(\Phi_N(t) (A))
&  \leq 
2^{p-1}
  \rho_{s, N, r}(A) + 2^{p-1}( C_{ r} t_r)^p\notag\\
&   \leq 
2^{p-1}  \rho_{s, N, r}(A) + 2^{-p-1} 
\notag\\
\intertext{by setting $p = p(\eps)= - \log_2 \eps $, } 
&   \leq 
2^{p(\eps)-1}
\dl + \tfrac 12 \eps \notag\\
& < \eps,
\end{align*}

\noi
by choosing $\dl = \dl(\eps) > 0$ sufficiently small.
This completes the proof of Proposition~\ref{PROP:meas3}.
\end{proof}

\subsection{Proof of Theorem \ref{THM:NLKG}}

We conclude this section by presenting the proof of Theorem~\ref{THM:NLKG}.
Proposition~\ref{PROP:meas3} implies that the truncated
weighted Gaussian measures $\rho_{s, N, r}$
are quasi-invariant under the truncated NLKG dynamics $\Phi_N(t)$
with the uniform control in $N\in \N$.
We first upgrade Proposition~\ref{PROP:meas3} to the untruncated 
weighted Gaussian measure $\rho_{s, r}$ 
defined in~\eqref{QFT1}.
Then, we exploit the mutual absolute continuity 
between $\rho_{s, r}$ and $\mu_{s, r}$, implying quasi-invariance
of $\mu_{s, r}$ under the full NLKG dynamics $\Phi(t) = \Phi_{\NKG}(t)$.
Finally, we conclude quasi-invariance of $\mu_s$ by taking $r \to \infty$.

\begin{lemma}\label{LEM:meas4}
Given $r > 0$, there exists $t_{ r} > 0$ such that 
given $\eps > 0$, 
there exists $\dl > 0$ such that 
if 
$$
\rho_{s,r} ( A)< \dl 
$$
 for a measurable set $A \subset {\mathcal H}^\sigma(\T^2)$, $\sigma\in (1,s)$,  
then we have 
$$
\rho_{s,  r} (\Phi(t) (A)) < \eps
$$
for any $t \in [0, t_{r}]$.
Note that $\dl > 0$ is independent of 
$t \in [0, t_{r}]$.
\end{lemma}
\begin{proof}
Let $t_r$ be as in  Proposition~\ref{PROP:meas3}.
We first consider the case when $A$ is compact in $\H^\s(\T^2)$. Let $\eps > 0$. 
Thanks to  Proposition~\ref{PROP:meas3}, there is $\delta_1>0$ such that if 
there exists $N_0 \in \N$ such that 
$$
 \rho_{s, N, r} ( A+B_{ \theta, \s})< \dl_1
$$
\noi
for any $N \geq N_0$, 
then we have 
\begin{align}
\rho_{s, N, r} (\Phi_N(t) (A+B_{\theta, \s})) < \frac{\eps}{2}
\label{delta_0}
\end{align}

\noi
for any $ t \in [0, t_{r}]$ and any  $N \geq N_0$.
Recall that $B_{\theta, \s}$ denotes the (closed) ball of radius $\theta > 0$ in $\H^\s(\T^2)$.

We now observe that there exist $\delta_2>0$,  $N_1 \in \N$, and $\theta>0$ such that if 
\begin{equation}\label{delta_1}
\rho_{s, N, r} ( A)< \dl_2
\end{equation}

\noi
for any $ N \geq N_1$,  
then we have 
\begin{equation}\label{delta_2}
 \rho_{s, N, r} ( A+B_{\theta, \s})< \dl_1
 \end{equation}
 
 \noi
for any $ N \geq N_1$.
  More precisely, by writing 
 $$
 d\rho_{s, N,r}(u,v) =G_N(u,v)d \mu_s(u,v)
 \qquad \text{and}\qquad  d\rho_{s,r}(u,v) =G(u,v)d \mu_s(u,v),
 $$

\noi
it follows from Proposition \ref{PROP:QFT} that 
$G_N$ converges to $G$ in $L^p(d\mu_s)$ for every $p<\infty$. 
We can therefore write
 \begin{align*}
  \rho_{s, N, r} ( A+B_{\theta, \s})
  & -\rho_{s, N, r} ( A)
  =\int_{A+B_{\theta, \s}}G_N d\mu_s-\int_{A}G_Nd\mu_s
  \\
&   =
  \int_{A+B_{\theta, \s}}(G_N-G)d\mu_s
  +\int (\ind_{A+B_{\theta, \s}} - \ind_A) Gd\mu_s
  + \int_{A}(G-G_N)d\mu_s.
 \end{align*}

\noi
 Now, for the first and  third terms,
  we use the convergence of $G_N$ to  $G$ in $L^1(d\mu_s)$,
   while, for the second term, we invoke the dominated convergence (here we used the fact that $A$ is closed).
Therefore, we conclude that  \eqref{delta_1} implies \eqref{delta_2}.
 We also observe that thanks to \eqref{rho_s},
  there exist $\delta>0$ and $N_2 \in \N$ such that if 
\begin{align}
 \rho_{s,r}(A)<\delta, 
\label{delta_3}
\end{align}

\noi
 then we have \eqref{delta_1}
 for any $N \geq N_2$.
At this point, we have  already fixed the values of $\delta$, $\theta$, $N_0$, $N_1$,  and $N_2$.
 Finally, 
 it  follows Lemma~\ref{LEM:approx} and \eqref{rho_s} that there exists $N_3  = N_3(t,\theta, \eps)\in \N$ such that
 if \eqref{delta_3} holds, then we have
 $$
\rho_{s, r} (\Phi(t) (A)) 
 \leq 
\rho_{s, r} (\Phi_N(t) (A+B_{\theta, \s})) 
 \leq \rho_{s, N, r} (\Phi_N(t) (A+B_{\theta, \s})) + \tfrac \eps 2 <\eps
$$

\noi
for any $t \in [0, t_r]$ and any $N\geq \max(N_0, N_1, N_2, N_3)$.
Here, we used \eqref{rho_s} and \eqref{delta_0}
in the second and third inequalities, respectively.
This completes the proof when $A$ is compact. 

We now prove the statement for arbitrary measurable sets.
 Once again, fix $\eps > 0$. 
 We have just proved that there is $\delta>0$ such that,
  for every compact set $K$ with $\rho_{s,r}(K)<\delta$, we have
\begin{equation}\label{FrBg} 
\rho_{s, r} (\Phi(t) (K)) <\frac{\eps}{2}
\end{equation}

\noi
for any $ t\in [0,t_r]$.
Now, let  $A$ be an arbitrary measurable set of $ {\mathcal H}^\sigma(\T^2)$, $\sigma\in (1,s)$,  
such that $\rho_{s,r}(A)<\delta$.
By the inner regularity of $\rho_{s, r}$,
there exists a sequence $\{K_j\}_{j \in \N}$
of compact sets such that $K_j \subset \Phi(t)( A)$ and 
\begin{equation}
 \rho_{s, r} (\Phi(t) (A)) = \lim_{j \to \infty} \rho_{s, r} (K_j).
\label{meas7}
 \end{equation}

\noi
Note that $\Phi(-t) (K_j)$ is compact since it is the image of 
the compact set $K_j$ under the continuous map $\Phi(-t)$.
Moreover, by the bijectivity 
of the flow $\Phi(-t)$, 
we have $\Phi(-t) (K_j) \subset \Phi(-t) \Phi(t) (A) = A$.
In particular, we have
 $ \rho_{s, r} (\Phi(-t) (K_j) ) < \dl$.
Then, applying \eqref{FrBg} 
for  the compact set $\Phi(-t) K_j$, we obtain
\begin{equation}
 \rho_{s, r} ( K_j) 
=   \rho_{s, r} \big( \Phi(t) (\Phi(-t)K_j) \big)
 <  \frac{\eps}{2}
\label{meas8}
 \end{equation}

\noi
for all $j \in \N$ and all $t \in [0, t_r]$.
Hence, the desired conclusion follows from \eqref{meas7} and \eqref{meas8}.
This completes the proof of Lemma~\ref{LEM:meas4}.
\end{proof}

Finally, we present the proof of  Theorem \ref{THM:NLKG}.
\begin{proof}[Proof of Theorem \ref{THM:NLKG}]
Let $ A\subset {\mathcal H}^\sigma(\T^2)$, $\sigma\in (1,s)$,  be a measurable set  such that $\mu_s(A) = 0$. Then, for any $r > 0$, we have 
\[\mu_{s, r}(A) = 0.\]

\noi
By the mutual absolute continuity 
 of $\mu_{s, r}$ and $\rho_{s, r}$, 
 we obtain
\[\rho_{s, r}(A) = 0.\]

\noi
Then, by Lemma \ref{LEM:meas4}, we have
\begin{align}
\rho_{s, r}(\Phi(t) (A)) = 0
\label{meas9}
\end{align}

\noi
for $t \in [0, t_r]$.
By iterating this argument, 
we conclude that \eqref{meas9} holds for any $t >0$.
By invoking the mutual absolute continuity 
 of $\mu_{s, r}$ and $\rho_{s, r}$ once again, 
 we have
\[\mu_{s, r}(\Phi(t) (A)) = 0.\]

\noi
Finally, the dominated convergence theorem yields
\[\mu_{s}\big(\Phi(t) (A)\big) 
= \lim_{r \to \infty} 
\mu_{s, r}\big(\Phi(t) (A)\big) = 0.
\]

\noi
By the time reversibility of the equation \eqref{KG-sys}, 
the same conclusion holds for any $t <0$.
This completes the proof of Theorem \ref{THM:NLKG}.
\end{proof}

\begin{remark}
\rm

By combining  Lemma \ref{LEM:meas2} 
with the Yudovich's argument \cite{Y} as in~\cite{Tzvet, OTz}
(but with the critical power $p^1$), 
we can obtain the following quantitative bound,
characterizing the quasi-invariance of $\rho_{s, r}$:
$$
\rho_{s, r} (\Phi(t) (A)) \lesssim 
\big(\rho_{s, r} ( A) \big)^{\frac {1}{c^{1+|t|}}}
$$

\noi
for any $t\in \R$.
Here, the constant $c = c(r)$ depends on $r > 0$.
\end{remark}

\section{Quasi-invariance under  the NLW dynamics}

As already mentioned, 
the proof of Theorem~\ref{THM:NLW} for the nonlinear wave equation is very close to 
that  of Theorem~\ref{THM:NLKG} that we just  presented in the previous section.
In this section,  we only explain the needed modifications.

\subsection{The modified Gaussian measures} 
Since the quadratic part of the Hamiltonian $H$ 
defined in~\eqref{H}
for  the nonlinear wave equation does not control the $L^2$-norm, we shall prove the quasi-invariance for a small modification of $\mu_s$  that is absolutely continuous with respect to $\mu_s$.

Define   $\wt \mu_{s} $ as  the induced probability measure under the map:
\begin{equation*}
\o \in \O \longmapsto (u^\o(x), v^\o(x))
 \end{equation*}
with 
\begin{equation*}
u^\o(x) = g_0 + \sum_{n \in \Z^2\setminus \{0\}} \frac{g_n(\o)}
{
(|n|^2+|n|^{2s+2})^{\frac{1}{2}}
}e^{in\cdot x}
\qquad \text{and}\qquad 
v^\o(x) = \sum_{n \in \Z^2} 
\frac{h_n(\o)}{
(1+|n|^{2s})^{\frac{1}{2}}
}e^{in\cdot x},
\end{equation*}
where  $\{ g_n \}_{n \in \Z^2}$ and $\{ h_n \}_{n \in \Z^2}$
are as in \eqref{series}.
With $\ft u(0) = \int_{\T^2} u\, dx$, we can formally write  $\wt \mu_{s} $ as 
$$
d \wt \mu_s = Z_s^{-1} 
e^{-\frac 12 \int v^2- \frac 12  \int (D^s v)^2 
-  \frac 12 (\int  u )^2 
-  \frac 12 \int  |\nb u|^2 
-  \frac 12 \int (D^{s+1} u)^2}du 
dv,
$$
where 
$$
D:= \sqrt{-\Dl}.
$$

\noi
As we shall see below, the expression
\begin{equation}\label{lili}
H_0(u, v) = \frac 12 \int_{\T^2} v^2+\frac 12  \int_{\T^2} (D^s v)^2 
+  \frac 12 \bigg(\int_{\T^2}  u\bigg)^2 + \frac 12 \int_{\T^2}  |\nb u|^2 + \frac 12 \int_{\T^2} (D^{s+1} u)^2
\end{equation}

\noi
appears as  the quadratic part of the renormalized energy in the context of the nonlinear wave equation.
We have the following statement.
\begin{lemma} \label{LEM:equiv}
Let $s > \frac 12$. Then, the Gaussian measures $\mu_s$ and $\wt \mu_s$ are equivalent.
\end{lemma}
\begin{remark}
{\rm
In view of Lemma~\ref{LEM:equiv},  
it suffices to study the quasi-invariance property of $\wt \mu_s$ under the flow of the  defocusing cubic nonlinear wave equation. 	
}
\end{remark}
\begin{proof}
Note that $\mu_s$ and $\wt \mu_s$
are product measures on $u$ and $v$.
Define (formally)  $\mu^1_s$ and $\mu^2_s$ by 
\begin{align*}
d \mu^1_s 
 = Z^{-1} e^{ -  \frac 12 \int (J^{s+1} u)^2}du 
\qquad \text{and}\qquad 
d \mu^2_s 
 = Z^{-1} e^{ -  \frac 12 \int (J^{s} v)^2}dv.
\end{align*}

\noi
Then, we have $\mu_s = \mu^1_s \otimes \mu^2_s$.
Similarly, by 
defining $\wt \mu^1_s$ and $\wt \mu^2_s$ by 
\begin{align*}
d \wt \mu^1_s 
&  = Z^{-1} e^{ -\frac 12 (\int u)^2-  \frac 12 \int  |\nb u|^2 
-  \frac 12 \int (D^{s+1} u)^2}du , \notag\\
d \wt \mu^2_s 
 & = Z^{-1} e^{ -\frac 12 \int v^2 -  \frac 12 \int (D^{s} v)^2}dv, 
\end{align*}

\noi
we have $\wt \mu_s = \wt \mu^1_s \otimes \wt \mu^2_s$.
Hence, it suffices to prove that $\mu^j_s$
and $\wt \mu^j_s$ are equivalent, $j = 1, 2$.

First, let us consider the $j = 1$ case.
Given $\s < s$, define $\ld_n$
and $\wt \ld_n$ by 
\[ \ld_n = \frac{1}{\jb{n}^{2s+2-2\s}}
\qquad\text{and}\qquad
\wt \ld_n = 
\begin{cases}
1, & \text{if } n = 0,\\
\frac{\jb{n}^{2\s}}{|n|^2+  |n|^{2s+2}}& \text{if } n \ne 0.
\end{cases}
\]

\noi
Then, $\mu_s^1$ and $\wt \mu_s^1$ are  the Gaussian measures on $H^\s(\T^2)$
with the covariance operators $Q$ and $\wt Q$ given by\footnote{Namely, $Q$ 
and $\wt Q$ are defined by
the following relations:
\[
-\frac 12 \int (J^{s+1} u)^2
= -\frac 12 \jb{ Q^{-1} u, u }_{H^\s}
\quad \text{and}
\quad
-\frac 12 \bigg(\int u\bigg)^2
-\frac 12 \int \big(u^2+   |\nb u|^2  + (D^{s+1} u)^2\big)
= -\frac 12 \jb{ \wt Q^{-1} u, u }_{H^\s}.
\]}
\[ Q e_n = \ld_n e_n
\qquad\text{and}\qquad
\wt Q e_n = \wt \ld_n e_n, 
\]

\noi
respectively, where $e_n(x) =  e^{ i n\cdot x}$.
Now, define $S_n$ by 
\[ S_n = \frac{(\ld_n- \wt \ld_n)^2}{(\ld_n + \wt \ld_n)^2}.\]

\noi
Then, by Kakutani's theorem \cite{Kakutani} (or Feldman-H\'ajek theorem \cite{Feldman, Hajek}),
it follows that $\mu^1_s$ and $\wt \mu^1_s$ are equivalent
if and only if 
\begin{align}
\sum_{n \in \Z^2} S_n < \infty.
\label{M3c}
\end{align}
	
\noi
Otherwise, they are singular.

For $n \ne 0$, we have
\begin{align*}
 S_n = \frac{(|n|^2+ |n|^{2s+2} - \jb{n}^{2s+2})^2}{(|n|^2 + |n|^{2s+2}+ \jb{n}^{2s+2})^2}
\sim  \frac{(|n|^2+ |n|^{2s+2} - (1+|n|^2)^{s+1})^2}{\jb{n}^{4s+4}}.
\end{align*}

\noi
By the mean value theorem applied to  $f(x) = x^{s+1}$, we have
\begin{align*}
\big| |n|^{2s+2} - (1+|n|^2)^{s+1}\big|
= |f(|n|^2)- f(1+|n|^2)\big|
\sim |n|^{2s}.
\end{align*}

\noi
Hence, we obtain
\begin{align*}
 S_n \sim  \frac{(|n|^2+ |n|^{2s})^2}{\jb{n}^{4s+4}}
\les \frac{|n|^{4}+  |n|^{4s}}{\jb{n}^{4s+4}}, 
\end{align*}

\noi
which is summable over $\Z^2$, provided that $s > \frac 12$.
This proves \eqref{M3c}
and the equivalence of $\mu^1_s$ and $\wt \mu^1_s$.
A similar computation yields
the equivalence of $\mu^2_s$ and $\wt \mu^2_s$.
We omit details.
\end{proof}

\subsection{Renormalized energy for NLW}
In this subsection, we derive the renormalized energy in the context of the truncated NLW: 
\begin{equation}\label{NLW-sys_N}
\begin{cases}
\partial_t u=v\\ \partial_t v=\Delta u-\pi_N((\pi_N u)^3).
\end{cases}
\end{equation}

\noi
Once the renormalized energy is derived,
 the remaining of the proof of Theorem~\ref{THM:NLW} is exactly the same as the proof of Theorem~\ref{THM:NLKG}. 

If $(u,v)$ is a solution to the truncated NLW  \eqref{NLW-sys_N}, 
then we have
\begin{equation}
\dt \bigg[\frac 12  \int_{\T^2} (D^s v_N)^2 + \frac 12 \int_{\T^2} (D^{s+1}  u_N)^2
\bigg] 
 =  \int_{\T^2} (D^{2s} v_N)(-u_N^3),
\label{HH1}
\end{equation}

\noi
where 
$(u_N,v_N)=(\pi_N u,\pi_N v)$ as before.
Let $s\geq 2$ be an even integer.
Then,  by  the Leibniz rule, we have
\begin{align}
\int_{\T^2} (D^{2s} v_N)(-u_N^3)
& =
-3\int_{\T^2} D^sv_N\, D^s u_N\, u_N^2\notag \\
& \hphantom{X} +
\sum_{\substack{ |\al|+|\be|+|\g|=s\\
|\al|,|\be|,|\g|<s
}}
c_{\alpha,\beta,\gamma}
\int_{\T^2}
D^sv_N\cdot \dd^\al u_N\cdot  \dd^\be u_N \cdot \dd^\g u_N
\label{HH2}
\end{align}

\noi
for some inessential constants $c_{\alpha,\beta,\gamma}$.
Furthermore,  we can write 
\begin{align}
-3\int_{\T^2} D^sv_N\, & D^s u_N\, u_N^2= - \frac 32 \dt \bigg[\int_{\T^2} (D^s u_N)^2u_N^2 \bigg]
+ 3 \int_{\T^2} (D^s u_N)^2 v_N u_N 
\notag \\
& =
  - \frac 32 \dt\bigg[ \int_{\T^2} \P_{\ne 0} [(D^su_N)^2] \,  \P_{\ne 0}[u_N^2] \bigg]
+ 3 \int_{\T^2} \P_{\ne 0}[(D^su_N)^2]\, \P_{\ne 0}[v_N  u_N]  
\notag \\
& \hphantom{X}
-   \frac 32 \dt \bigg[\int_{\T^2} (D^s u_N)^2 \int_{\T^2}  u_N^2\bigg]
+   3  \int_{\T^2} (D^s u_N)^2\int  v_N  u_N .
\label{HH3}
\end{align}

\noi
As in \eqref{H1}, 
 the last two terms on the right-hand side  are problematic. 
 Therefore, we once again introduce  a suitable renormalization. Define $\wt \s_N$ by
\[ 
\wt \s_N = 
\E_{\wt \mu_s}
\bigg[\int_{\T^2} (D^s \pi_{N}u)^2\bigg] 
= \sum_{\substack{n \in \Z^2\\1\leq |n|\leq N}} 
\frac{|n|^{2s}}{|n|^2+|n|^{2s+2}}\sim \log N.
\]

\noi
Then, we have
\begin{align}
 -   \frac 32 & \dt \bigg[\int_{\T^2} (D^s u_N)^2 \int  u_N^2\bigg]
+   3  \int_{\T^2} (D^s u_N)^2\int  v_N u_N \notag\\
& = 
 -   \frac 32 \dt \bigg[\bigg(\int_{\T^2} (D^s u_N)^2-  \wt \s_N\bigg)\int_{\T^2}  u_N^2\bigg]
+   3 \bigg( \int_{\T^2} (D^s u_N)^2 - \wt \s_N\bigg)\int  v_N  u_N.
\label{HH4}
\end{align}

\noi
Thanks to the Wiener chaos estimate (Lemma \ref{LEM:hyp}), the term 
$$
\int_{\T^2} (D^s u_N)^2-  \wt\s_N
$$
enjoys the bound
$$
\bigg\|
\int_{\T^2} (D^s \pi_N u)^2- \wt \s_N
\bigg\|_{L^p(d \wt \mu_s(u,v))}\leq Cp, 
$$

\noi
for any finite $p \geq 2$, 
where the constant $C>0$ is independent of $p$ and $N$.

We now define the renormalized energy $H_{s, N}(u,v)$ by 
\begin{equation*}
H_{s, N}(u,v)  = \frac 12  \int (D^s v)^2 + \frac 12 \int (D^{s+1}  u)^2+ 
\frac 32  \int (D^s \pi_N u)^2  (\pi_N u)^2 - \frac 32 \, \wt\s_N  \int (\pi_N u)^2.
\end{equation*}

\noi
Then, it follows from \eqref{HH1} - \eqref{HH4} that, 
 if $(u,v)$ is a solution to \eqref{NLW-sys_N},  then we have
\begin{align}
\dt H_{s, N}(u_N,v_N) 
&  =  
 3 \int_{\T^2}  \P_{\ne 0}[ (D^s u_N)^2  ]  \, \P_{\ne 0} [v_N  u_N ]
+   3 \bigg( \int_{\T^2} (D^s u_N)^2 - \wt \s_N\bigg)\int_{\T^2}  v_N  u_N 
\notag \\
& \hphantom{X}
+ 
\sum_{\substack{ |\al|+|\be|+|\g|=s\\
|\al|,|\be|,|\g|<s
}}
c_{\al,\be,\g}
\int_{\T^2}
D^sv_N\cdot \dd^\al u_N \cdot \dd^\be u_N \cdot \dd^\g u_N.
\label{H8_pak}
\end{align}

\noi
As in Subsection \ref{SUBSEC:1.4}, 
all terms on the right-hand-side of \eqref{H8_pak} are suitable for a perturbative analysis.
However,  a modification of the quadratic part is needed 
in order to have a resulting measure absolutely continuous with respect to $\wt \mu_s$.

For this purpose, we define the full renormalized  energy $E_{s, N}(u,v)$ as
\begin{align}
E_{s, N}(u,v)=H_{s, N}(u,v)+H_N(u, v) + \frac 12 \bigg(\int u \, dx\bigg)^2,
\label{NEW0}
\end{align}

\noi
where $H_N$ is the conserved energy for the truncated NLW \eqref{NLW-sys_N} defined 
by 
\begin{equation*}
H_N(u,v):=\frac{1}{2}\int_{\T^2}\big(|\nabla u|^2 + v^2\big)dx +\frac{1}{4}\int_{\T^2}(\pi_N u)^4dx.
\end{equation*}

%
\noi
The quadratic part of $E_{s,N}$ is now given by \eqref{lili}, 
resulting in the Gaussian measure $\wt \mu_s$ equivalent to $\mu_s$. 
Using  the truncated NLW ~\eqref{NLW-sys_N}, we have that
\begin{align*}
\dt E_{s, N}(u_N,v_N) 
& =\dt H_{s, N}(u_N,v_N) +\bigg(\int_{\T^2} u_N \bigg) \bigg(\int_{\T^2} v_N \bigg).
\end{align*}

\noi
Hence, the only new term  to be handled as compared to the proof of Theorem~\ref{THM:NLKG} is
\begin{equation}\label{NEW}
\bigg(\int_{\T^2} u_N \bigg) \bigg(\int_{\T^2} v_N \bigg).
\end{equation}
More precisely, we need to estimate \eqref{NEW} under the restriction
on the truncated energy
\begin{equation}\label{restriction}
H_N(u,v)
\leq r.
\end{equation}

\noi
By the compactness of the domain $\T^2$, 
we have 
$$
\bigg|\bigg(\int_{\T^2} u_N \bigg) \bigg(\int_{\T^2} v_N \bigg)\bigg|
\leq \|\pi_N u\|_{L^4(\T^2)}
\|\pi_N v\|_{L^2(\T^2)}\leq C_r
$$

\noi
 under \eqref{restriction}.
Therefore,  the contribution of \eqref{NEW} to $\dt E_{s, N}(u_N,v_N)$ is easy to deal with. 
We finally note that the introduction of $H_N(u,v)$ in the definition \eqref{NEW0} of the modified energy leads 
to the introduction of a new harmless term $\int (\pi_N u)^4$ in the definition of the weighted 
Gaussian measures $\rho_{s,N, r}$.
The remaining part of the analysis leading to the proof of Theorem~\ref{THM:NLW} 
is exactly the same\footnote{Note that the proof of the change-of-variable formula (an analogue of Lemma \ref{LEM:cov} for NLW)
requires (i) the Hamiltonian structure of the truncated dynamics \eqref{NLW-sys_N},
leading to the invariance of the Lebesgue measure $L_N$ on 
$\EE_N\times \EE_N$ and (ii) invariance of the marginal Gaussian measure $\wt \mu_{s, N}^\perp$
on $\pi_N^\perp \H^\s(\T^2)$.
See the proofs of Proposition 4.1 in \cite{Tzvet} and Proposition 6.6 in \cite{OTz}.
Clearly, (i) is satisfied. We  see that (ii) is also satisfied
since $H_0$ defined in \eqref{lili} satisfies
\[H_0(\pi_N^\perp u, \pi_N^\perp v)
  = \frac 12 \int_{\T^2} (\pi_N^\perp v)^2+\frac 12  \int_{\T^2} (D^s \pi_N^\perp v)^2 
 + \frac 12 \int_{\T^2}  |\nb \pi_N^\perp u|^2 + \frac 12 \int_{\T^2} (D^{s+1} \pi_N^\perp u)^2\]
 
 \noi
 which is conserved by the linear wave dynamics on the high frequencies 
$(\EE_N\times \EE_N)^\perp$.
} 
as the one already presented in the proof of Theorem~\ref{THM:NLKG} and therefore  
we omit details.

\appendix

\section{On the dispersion generalized NLKG}
\label{SEC:disp}

In this appendix, we briefly discuss the situation for
 the  (much easier) dispersion generalized NLKG \eqref{NLKG5} with $\be > 1$.
The equation \eqref{NLKG5} is a Hamiltonian equation with 
the Hamiltonian given by
\begin{align*}
E^\be(u) 
& = \frac 12 \int_{\T^2} (J^\be u )^2 + \frac 12 \int_{\T^2} v^2 + \frac14 \int_{\T^2} u^4. 
\end{align*}

\noi
By repeating the computation in Subsection \ref{SUBSEC:1.4}, we have
\begin{align}
\dt \bigg[\frac 12    \int_{\T^2} (J^s v)^2  & + \frac 12 \int_{\T^2} (J^{s+\be}  u)^2
\bigg] 
 =  \int_{\T^2}  J^{2s} v (- u^3 )
\notag\\
& =  - 3\int_{\T^2} (\dt J^{s} u)J^s u \cdot u^2 + \text{l.o.t.}
\notag\\
& =  - \frac 32 \dt \bigg[\int_{\T^2} (J^su)^2u^2 \bigg]
+ 3 \int_{\T^2} (J^su)^2\dt u \cdot u  + \text{l.o.t.}, 
\label{G10}
\end{align}

\noi
\noi
where ``l.o.t.''\,denotes various (insignificant) lower order terms.
Define $E^\be_s(u, v)$ and $E^\be_{s, N}(u, v)$by 
\begin{align}
E^\be_s(u, v) 
& = 
\frac 12  \int (J^s \dt u)^2 + \frac 12 \int (J^{s+\be}  u)^2
+ \frac 32  \int  (J^s u)^2u^2, 
\label{G11}\\
E^\be_{s, N}(u, v) 
& = 
\frac 12  \int (J^s \dt u)^2 + \frac 12 \int (J^{s+\be}  u)^2
+ \frac 32  \int  (J^s \pi_N u)^2 (\pi_N u)^2.
\label{G12}
\end{align}

Define the following weighted Gaussian measure $\rho^\be_{s,N, r}$, $N \in \N \cup\{\infty\}$,  by 
\[ d \rho^\be_{s,N, r} = Z^{-1} \ind_{\{E_N^\be(u, v) \leq r\}} e^{-E_{s, N}^\be(u, v)} du dv
= Z^{-1}  \ind_{\{E_N^\be(u, v) \leq r\}} e^{
- \frac 32  \int  (J^s \pi_N u)^2 (\pi_N u)^2} d\mu^\be_s,\]

\noi
where $\mu^\be_s$ is as in \eqref{mu_b} and $E_N^\be$ is the truncated energy defined by
\begin{align*}
E_N^\be(u) 
& = \frac 12 \int_{\T^2} (J^\be u )^2 + \frac 12 \int_{\T^2} v^2 + \frac14 \int_{\T^2} (\pi_N u)^4. 
\end{align*}

\noi
Then, in view of the comment in Remark~\ref{REM:disp}, 
we can  repeat the argument in Section~\ref{SEC:typical}
(without any renormalization)
and show that $\rho^\be_{s,N, r}$ is a well defined
probability measure (even when $r = \infty$
thanks to the defocusing nature of the equation)
with a uniform bound in 
 $N \in \N \cup\{\infty\}$.

Let us now turn to the energy estimate.
Let $s \geq \be > 1$.
It follows from~\eqref{G10},~\eqref{G11},  and~\eqref{G12} that 
\begin{align}
\dt E_{s, N}^\be(u, v) 
& =  
3 \int (J^s \pi_N u)^2  \cdot \pi_N v \cdot \pi_N u  + \text{l.o.t.}
\label{G13}
\end{align}

\noi
for a solution $(u, v)$ to the following truncated dispersion generalized NLKG:
\begin{equation*}
\begin{cases}
\partial_t u=v\\
\partial_t v=J^{2\be}u -\pi_N((\pi_N u)^3).
\end{cases}
\end{equation*}

\noi
By interpolation and the Sobolev embedding $H^\be(\T^2) \subset L^\infty(\T^2)$, $\be > 1$, we have
\begin{align*}
\bigg| \int_{\T^2} (J^s \pi_N u)^2\cdot \pi_Nv  \cdot \pi_N u\bigg|
& \leq \| J^s \pi_N u\|_{L^4_x}^2 \underbrace{\| \pi_N v \|_{L^2_x} \| \pi_N u \|_{L^\infty_x}}_{\les E_N^\be (u, v)}\\
& \les \| J^{s+\be - 1- \eps} \pi_N u\|_{L^r_x}^{2(1-\theta)}
\big( E_N^\be(u, v)\big)^{1+2\theta},
\end{align*}

\noi
for some  $\theta \in (0, 1]$ and $r > 4$ satisfying
\begin{align*}
s  = \theta \be + (1-\theta) ( s+ \be - 1- \eps)
\qquad \text{and}\qquad
\frac{1}{4}  = \frac{\theta}{2} + \frac{1-\theta}{r}.
\end{align*}

\noi
Hence, by the Wiener chaos estimate (Lemma \ref{LEM:hyp}), 
we obtain the  crucial  energy estimate:
\begin{align*}
\bigg\|  \ind_{\{E_N^\be(u, v) \leq r\}} \cdot \int_{\T^2} (J^s \pi_N u)^2\cdot \pi_Nv  \cdot \pi_N u
\bigg\|_{L^p(d\mu^\be_{s})}
& \les p^{1-\theta}
\end{align*}

\noi
for some $\theta > 0$.
The lower order terms in~\eqref{G13} can be handled in a similar (or easier) manner.
Then, one can repeat the argument in~\cite{Tzvet}
and prove quasi-invariance of the Gaussian measure $\mu_s^\be$,
at least for an even integer $s \geq \be$.

\begin{ackno}\rm
T.O.~was supported by the European Research Council (grant no.~637995 ``ProbDynDispEq'').
The authors would like to thank Prof.~Andrew Stuart 
for pointing out the remark on propagation of additional regularity in Subsection \ref{SUBSEC:1.3}, 
Prof.~David Elworthy
for pointing out the references \cite{P1, P2}, 
and Prof.~Sergey Nazarenko
for an interesting discussion on wave turbulence.
The authors are also grateful to the anonymous referees
for their helpful comments that have improved the presentation of this paper.

\end{ackno}

\end{document}